\documentclass[preprint,a4paper, 11pt, authoryear]{elsarticle} 
\usepackage{geometry}
\usepackage{amsmath}
\usepackage{amssymb}
\usepackage{amsthm}
\usepackage{amsfonts}
\newtheorem{proposition}{Proposition}

\newtheorem{theorem}{Theorem}
\newtheorem{definition}{Definition}
\usepackage{algorithm}
\usepackage{xcolor}
\usepackage{algpseudocode}
\usepackage{mathtools}
\DeclarePairedDelimiter{\ceil}{\lceil}{\rceil}
\DeclarePairedDelimiter{\floor}{\lfloor}{\rfloor}

\usepackage{array,multirow,graphicx,booktabs,longtable}
\usepackage{pgfplots}
\pgfplotsset{compat=1.18}
\makeatletter
\renewenvironment{proof}[1][\proofname]{%
  \par\pushQED{\qed}\normalfont%
  \topsep3pt\relax 
  \trivlist\item[\hskip\labelsep\itshape #1\@addpunct{.}]%
  \ignorespaces
}{%
  \vskip-0.75cm 
  \popQED\endtrivlist\@endpefalse
}
\makeatother
\usepackage{tabularx}
\usepackage{cuted}
\usepackage{flushend}
\usepackage{float}
\usepackage{natbib}
\usepackage[hidelinks]{hyperref}

\begin{document}

\begin{frontmatter}

\title{On Branch-and-Price for Project Scheduling}

\author[inst1]{Maximilian Kolter$^\ast$}
\author[inst1]{Martin Grunow}            
\author[inst1]{Rainer Kolisch}

\affiliation[inst1]{organization={TUM School of Management, Department of Operations \& Technology},
            addressline={Arcisstraße 21}, 
            city={Munich},
            postcode={80333}, 
            country={Germany}}

\begin{abstract}
Integer programs for resource-constrained project scheduling problems are notoriously hard to solve due to their weak linear relaxations. Several papers have proposed reformulating project scheduling problems via Dantzig-Wolfe decomposition to strengthen their linear relaxation and decompose large problem instances. The reformulation gives rise to a master problem that has a large number of variables. Therefore, the master problem is solved by a column generation procedure embedded in a branching framework, also known as branch-and-price. While branch-and-price has been successfully applied to many problem classes, it turns out to be ineffective for most project scheduling problems. This paper identifies drivers of the ineffectiveness by analyzing the structure of the reformulated problem and the strength of different branching schemes. Our analysis shows that the reformulated problem has an unfavorable structure for column generation: It is highly degenerate, slowing down the convergence of column generation, and for many project scheduling problems, it yields the same or only slightly stronger linear relaxations as classical formulations at the expense of large increases in runtime. Our computational experiments complement our theoretical findings. 
\end{abstract}


\end{frontmatter}

\section{Introduction}\label{sec::Introduction}
The resource-constrained project scheduling problem (RCPSP) is a well-known NP-hard optimization problem with a wide range of applications. Its classic version aims at scheduling a number of activities subject to resource and precedence constraints such that the makespan is minimized. Over the years, researchers have proposed many extensions, generalizations, and algorithms for the RCPSP and its variants \citep{hartmann2010, hartmann2022, schwindt2015a}.
As recently noted by \cite{bredael2023}, despite decades of research and an active project scheduling community, most research focuses on heuristics, whereas exact approaches capable of solving large instances are lacking. The most popular exact approach for project scheduling consists of formulating the problem as an integer program and solving it using a commercial solver. On the one hand, this approach can be used for a wide range of problem variants and is easy to implement. On the other hand, it quickly becomes intractable for larger problem instances. This behavior is not unique to project scheduling but is common to many optimization problems. As a remedy for many hard problems, such as vehicle routing \citep{costa2019} and bin packing \citep{sadykov2013}, researchers have developed exact \textit{branch-and-price} algorithms capable of solving large problem instances \citep{desrosiers2024}.

Branch-and-price refers to branch-and-bound algorithms designed to solve problems with a large number of columns (variables) that often arise from reformulating a compact formulation via \textit{Dantzig-Wolfe decomposition} \citep{dantzig1960}. The reformulated problem, commonly called \textit{master problem}, is often too large to be solved directly. Therefore, branch-and-price uses \textit{column generation} to solve the linear relaxation at each node in the branching tree. Column generation iteratively solves a \textit{restricted master problem} that contains only a small subset of columns and generates improving columns on demand via a \textit{pricing problem}, which is itself an optimization problem. Branch-and-price is motivated by several benefits \citep{barnhart1998}. First, the linear relaxation of the reformulated problem is often much tighter. Second, the problem may be decomposed such that the resulting pricing problem can be solved efficiently by a problem-specific algorithm.  Third, for problems with block-angular structure, the pricing problem can be decomposed into several independent problems. Furthermore, if the pricing problems are identical, one can break symmetries by aggregating them.

Motivated by these advantages, branch-and-price and other column-generation-based algorithms have also been proposed for project scheduling problems, e.g., by \cite{deckro1991}, \cite{brucker2000}, and \cite{vanEeckhout2020}. However, convincing experimental results demonstrating their superiority compared to directly solving the compact formulation are missing. This leads to the hypothesis that branch-and-price algorithms perform poorly for most project scheduling problems. Until now, no compelling explanation for this hypothesis has been presented. The main contribution of this paper consists of theoretical and numerical results that offer the missing explanation for the ineffectiveness of most column-generation approaches proposed for project scheduling. Specifically, our contribution is four-fold:

\begin{enumerate}
    \item We provide a comprehensive survey of column generation and branch-and-price algorithms for project scheduling. 
    \item We present a problem definition and corresponding Dantzig-Wolfe decomposition that generalizes most formulations considered in the context of column generation and branch-and-price for project scheduling.
    \item We analytically and numerically investigate structural properties of the Dantzig-Wolfe decomposition for project scheduling and their impact on column generation. This analysis provides evidence as to why many column-generation algorithms for project scheduling are ineffective. 
    \item We analyze the strength of the proposed branching strategies in the context of branch-and-price for project scheduling, which provides further support for our hypothesis.
\end{enumerate}

The remainder of this paper is organized as follows. First, we explain the technical details of branch-and-price algorithms necessary for the understanding of this work (Section \ref{sec::BPBasics}). Second, we survey the literature regarding column generation algorithms for project scheduling (Section \ref{sec::LiteratureReview}). Then, we provide a Dantzig-Wolfe decomposition for project scheduling, which generalizes many of the approaches from the literature (Section \ref{sec::DantzigWolfe}). We analyze the Dantzig-Wolfe decomposition by studying the structure of the master (Section \ref{sec::MasterProblem}) and pricing problems (Section \ref{sec::PricingProblems}), including a discussion of why they cause convergence issues and lead to weak lower bounds. Then, we assess branching rules from the literature regarding their impact on the relaxation strength (Section \ref{sec::Branching}). Finally, we provide a synthesis of our results (Section \ref{sec::Synthesis}) before closing with a brief conclusion (Section \ref{sec::Conculsion}). 

\section{Branch-and-Price}\label{sec::BPBasics}
In this section, we explain the fundamentals of Dantzig-Wolfe decomposition and branch-and-price by outlining its two main algorithmic components, column generation and branching.

\subsection{Dantzig-Wolfe Decomposition}
For the illustration of Dantzig-Wolfe decomposition consider the following \textit{compact} integer program (see \cite{desrosiers2024}):
\begin{align}
    z_{IP} = \min ~& c^{T}x \tag{IP} \label{compact}\\
    \text{s.t.}~& Ax \leq a &&  \notag\\
        & Bx \leq b && \notag\\
        & x \in \mathbb{Z}_{\geq 0} \notag
\end{align}
Note that we omit the dimensions of matrices and vectors but assume they are of the appropriate sizes. Furthermore, let us assume that $A$ has a block-angular structure 
$$
A = \begin{bmatrix}
A_{1} &  & &  \\
& A_{2} &  & \\
&  & \ddots & \\
&  &  & A_{N}  \\
\end{bmatrix} \text{and }
a = \begin{bmatrix}
a_{1}   \\
a_{2}  \\
\vdots  \\
 a_{N}  \\
\end{bmatrix} 
$$
with $N$ blocks. Similarly, let $c_n$, $x_n$, and $B_n$ be the cost coefficients, variables, and submatrices of $B$ corresponding to the columns of $A_n$ for $n=1,\ldots, N$.
We may group the constraints of the integer program \ref{compact} into the following (nonempty) subsets:
\begin{align}
    &\mathcal{A}_{n} := \{x_n \in \mathbb{Z}_{\geq 0}: A_{n}x_{n} \leq a_{n}\} &&\forall n=1,\ldots,N\\
    &\mathcal{B} := \{x \in \mathbb{Z}_{\geq 0}: Bx \leq b\}
\end{align}
Assuming that the convex hull $conv(\mathcal{A}_{n})$ of $\mathcal{A}_{n}$ is bounded for all $n = 1, \ldots, N$, we can reformulate the compact \ref{compact} via \textit{discretization} into the following equivalent \textit{extensive formulation} or integer \textit{master problem} (IMP):

\begin{align}
    z_{IMP} =  \min~ & \sum_{n=1}^{N}\sum_{\omega \in \Omega_{n}}(c^{T}_n x_{n \omega})\lambda_{n \omega} \tag{IMP}\label{extensive}\\
    \text{s.t.}~& \sum_{\omega \in \Omega_{n}} \lambda_{n \omega} = 1 && \forall n=1,\ldots,N && [\pi_{n}] \notag\\
        & \sum_{n=1}^{N}(B_{n} x_{n\omega})\lambda_{n \omega} \leq b &&  && [\pi_{b}] \notag\\ 
        & \lambda_{n \omega} \in \mathbb{Z}_{\geq 0}&&\forall n=1,\ldots,N,~ \omega \in \Omega_{n}\notag
        \notag
\end{align}
Where $\lambda_{n \omega}$ is a decision variable indicating whether we select the partial solution $x_{n \omega}$ for the subset of constraints $\mathcal{A}_{n}$ out of the set of all feasible solutions $\omega \in \Omega_{n}$ for $\mathcal{A}_{n}$. Furthermore, let  $\pi_{a}$  and  $\pi_{b}$  be the dual vectors corresponding to the linear relaxation of the \ref{extensive}. A major benefit of this reformulation is that it may yield smaller feasible regions and, thus, tighter bounds. The downside of Dantzig-Wolfe decomposition is that sets $\Omega_{n}$ may become so large that directly solving the master problem is impractical, and one may need to resort to column generation. 

\subsection{Column Generation}
To apply column generation to the master problem, a so-called \textit{restricted master problem} (RMP) is created by relaxing the integrality conditions and limiting the set of variables (columns) to a manageable subset $\lambda_{n \omega}  \in \Tilde{\Omega}_n \subset \Omega_n$. 
Then, the RMP is solved, and its dual values $\pi_{n}$ and $\pi_{b}$ are used to check if a column, which is not considered in the RMP yet, has negative-reduced costs $\hat{c}_{n}(\pi_{n},\pi_{b})$:   
\begin{align}
    \hat{c}_{n}(\pi_{n},\pi_{b}) = -\pi_{n} + \min_{x_n \in \mathcal{A}_n}(c^T_n- \pi_{b}^TB_n)x_n \label{reducedCosts}
\end{align}
Again, explicitly evaluating each potential column may be impractical. Hence, a pricing problem, which is itself an optimization problem, is used to implicitly check if a column with negative reduced costs exists. Due to the block-angular structure, we obtain $N$  independent pricing problems that seek to generate new columns with minimum reduced costs: 

\begin{align}
    \min &~\hat{c}_{n}(\pi_{n},\pi_{b})  \tag{PP}\label{pricing}\notag\\
     \text{s.t.}~& A_n x_n \leq a_n  \notag \\
        & x_n \in \mathbb{Z}_{\geq 0}  \notag
\end{align}
If the pricing problem for some $n \in \mathcal{N}$ returns a solution $x^{\ast}_{n}$ with negative reduced costs, its corresponding column $\lambda^{\ast}_{n \omega}$  is added to the RMP, i.e., $\Tilde{\Omega}_n \gets \Tilde{\Omega}_n \cup \{\omega\}$, and the process is repeated. Otherwise, the RMP's current solution is the optimal solution to the linear relaxation of the master problem, and the algorithm terminates.

\subsection{Branching}\label{sec::BranchingBasics}
To solve the master problem to integrality, branch-and-price algorithms embed column generation into a branching framework. Branch-and-price works like classical branch-and-bound, with the difference that the nodes of the branching tree are solved via column generation. If a node returns a fractional solution, a branching decision is made on how to partition the problem into two or more subproblems that exclude the fractional solution. In column generation, branching decisions can be enforced in multiple ways, belonging to mainly two groups; see \cite{vanderbeck2000} and \cite{desrosiers2024}. The first group enforces branching decisions by branching on the master variables, i.e., by adding branching constraints to the master problem. Adding constraints to the master problem creates new dual values and impacts the computation of reduced costs. Hence, the objective function of the pricing problem must be adjusted accordingly. The second group, known as branching on original variables, removes columns that violate the branching decisions from the master problem and adds branching constraints to the pricing problems to avoid the generation of violating columns. Branching on the original variables has the advantage that branching rules proposed for the original formulation can be used. Furthermore, as no constraints are added to the master problem, no new dual values emerge. However, adding constraints to the pricing problems may break their structure.

\section{Literature Review} \label{sec::LiteratureReview}
In the following, we survey the literature regarding column-generation approaches for project scheduling. First, we discuss the considered problem variants, the proposed Dantzig-Wolfe decompositions, and column generation algorithms before setting our work in context to the literature. 

\subsection{Problem Variants and Formulations}
We group the considered RCPSP variants according to the structure and terminology of the  survey papers by \cite{hartmann2010, hartmann2022}. Some studies combine several problem variants, but for conciseness, we discuss details of the master and pricing problems only within the extent of the variants that drive the structure of their Dantzig-Wolfe decompositions. A summary of the considered RCPSP variants is given in Table \ref{table::LiteratureOverview}.

\begin{table*}[!h]
\renewcommand*{\arraystretch}{0.75}
\footnotesize
\centering
\begin{tabular}{llccccccc}\toprule
\multirow{3}{*}{Publication}     & \multirow{2}{*}{Problem} &   \multirow{2}{*}{Multi-}                   & \multicolumn{2}{c}{Activities}    & \multicolumn{3}{c}{Resources}              \\ \cmidrule(rl){4-5}\cmidrule(rl){6-8}
                                                                                 &                         Variant      & Project & {Pre-} & Multi- & \multirow{2}{*}{Dedicated}                        & Multi- &    Time-Varying                                       \\
                                                                                   &                               &                         &           emptive                   & Mode  &    &                  Skilled    & Capacities                                    \\
\midrule
\cite{deckro1991}                                 & RCMPSP                                         & $\blacksquare$               &                             &       &            &               &                          &                                     \\       
\cite{brucker2000}                              & PRCPSP                                    &                          & $\blacksquare$                         &            &               &                 &                                                  \\
\cite{brucker2003}                              & PRCPSP                                    &                          & $\blacksquare$                         &  \Large$\bullet$          &               &                 &                                                 \\
\cite{drexl2001}                   & RIP                                &                          &                             &       &                           &                 &                                                 \\
\cite{akkan2005}                   & PSP                               &                          &                             &   \Large$\bullet$            &               &                 &                                                  \\
\cite{montoya2014}                & MSPSP                                     &                          &                             &       &                            $\blacksquare$      & $\blacksquare$   &                                             \\
\cite{coughlan2015}                               & MRIP                               &                          &                          &   & $\blacksquare$               &                              & $\blacksquare$       \\
\cite{moukrim2015}                                & PRCPSP                                    &                          & $\blacksquare$                         &            &               &                 &                                                  \\
\cite{volland2017}                              & RCPSP                       &                          &                             &       & \Large$\bullet$          &                            &              \Large$\bullet$ \\
\cite{wang2019}              & RCPSP                     &                          &                             &     &       $\blacksquare$        &               &                       \\
\cite{vanEeckhout2020}                    & MRCPSP                           &                          &                             & \Large$\bullet$    &     \Large$\bullet$       &                             &              \Large$\bullet$    \\         
\bottomrule
\end{tabular}
\caption{RCPSP variants considered in the context of column generation. Both $\blacksquare$ and \begin{Large}$\bullet$\end{Large} indicate the presence of a problem characteristic, whereas $\blacksquare$ indicates that a characteristic is exploited to decompose the problem.}\label{table::LiteratureOverview}
\end{table*}

\subsubsection{Multi-Project}
\cite{deckro1991} propose a Dantzig-Wolfe decomposition for the resource-constrained multi-project scheduling problem (RCMPSP) that exploits its structure by decomposing the problem by projects. The proposed master problem selects a combination of project schedules (columns) out of the set of all potential precedence-feasible project schedules such that their combination is resource-feasible. Because the set of all project schedules is prohibitively large, the schedules are generated dynamically via column generation. The pricing problems, one for each project, are project scheduling problems (PSPs) without resource constraints that seek to generate improving schedules. 

\subsubsection{Generalized Activity Concepts}
\paragraph{Preemptive Scheduling}
In contrast to the classical RCPSP, the preemptive RCPSP (PRCPSP) allows interrupting and restarting activities. \cite{brucker2000, brucker2003}, and \cite{moukrim2015} propose column generation for the \textit{antichain}-based PRCPSP formulation by \cite{mingozzi1998}. An antichain is a set of activities that can be processed simultaneously without violating resources or precedence constraints. All three papers consider a master problem that selects antichains, represented by columns, for different time intervals. Resource constraints are implicitly considered as only one antichain may be active at each point in time. The pricing problems used to generate improving antichains can be interpreted as knapsack problems with side constraints. The antichain formulation allows decomposing the pricing problem by time. Furthermore, for different time intervals, the pricing problems may be identical, allowing to break symmetries via aggregation.  

\paragraph{Multiple Modes}
The multi-mode RCPSP (MRCPSP) assumes that activities can be processed in several modes \citep{talbot1982}, where the modes represent different alternatives for processing activities that differ in resource requirements, duration, and costs. 
\cite{akkan2005} study a time-cost trade-off problem modeled as multi-mode PSP (MPSP) without resource constraints that seeks to find a minimum cost schedule for a given deadline. The authors propose decomposing the set of activities into several subsets, where a partial schedule for a subset of activities represents a column in the master problem. The master problem then selects a cost-optimal schedule by combining partial schedules. The pricing problems are again MPSPs, where each pricing problem corresponds to one subset of activities. Other studies that consider multiple modes in the context of column generation are \cite{brucker2003}, \cite{coughlan2015}, and \cite{vanEeckhout2020}.

\subsubsection{Generalized Resource Constraints}

\paragraph{Dedicated Resources}
A resource is dedicated either because it is unary, meaning it can only process one activity at a time, or because an activity requests exactly one resource. \cite{wang2019} and \cite{coughlan2015} propose to exploit dedicated resources by applying a Dantzig-Wolfe reformulation. Both papers consider a master problem that selects schedules for each dedicated resource subject to precedence constraints. \cite{wang2019} consider an RCPSP with classical and unary resources, assuming that each activity requires multiple resources of which exactly one is unary.  \cite{coughlan2015} consider dedicated resources in the sense that each activity requires exactly one resource. In both cases, the problem can be decomposed such that each dedicated resource has its own pricing problem. Dedicated resources are also considered in the study by \cite{montoya2014}, which we discuss in the following.

\paragraph{Resources with multiple skills}
In the multi-skill RCPSP (MSRCPSP), resources have multiple skills, and activities require certain skills to be processed. While processing an activity may require multiple skills, the MSRCPSP usually assumes that resources are unary, meaning they can only contribute with one skill to one activity at a time. \cite{montoya2014} proposes a formulation of the MSRCPSP that exploits this structure. The master problem selects columns that represent processing patterns, i.e., start time and resources assignments to activities. The pricing problems, which seek to generate improving patterns, can be decomposed by time and activity resulting in min-cost max-flow problems. 

\paragraph{Resource capacities varying with time}
The basic RCPSP assumes constant resource capacities. However, in applications such as the integrated shift and project scheduling problem considered by \cite{volland2017} and \cite{vanEeckhout2020}, where the resource availability depends on the shift schedules of the personnel, the resource capacity may vary over time. Moreover, in this setting, the resource capacity is a decision rather than a parameter. \cite{volland2017} and \cite{vanEeckhout2020} both propose a Dantzig-Wolfe decomposition with a master problem that selects shift schedules for the workers and a project schedule for the activities balancing resource supply and demand. This formulation has two column types, one representing the shifts and the other representing the project schedules. Two pricing problems result, one identifies improving shift schedules and the other identifies improving project schedules. We focus on the project scheduling component of their models. \cite{volland2017} considers the classic single-mode RCPSP, resulting in a PSP pricing problem as considered by \cite{deckro1991}. \cite{vanEeckhout2020} consider the more general multi-mode setting, resulting in an MPSP pricing problem as considered by \cite{akkan2005}.  Another application of resources with varying availability is turnaround scheduling with calendar constraints, as considered by \cite{coughlan2015}. Calendars define cyclic time intervals during which a resource is available or absent, such as human resources that work five consecutive days followed by two days off. \cite{coughlan2015} exploits this structure, decomposing the problem based on calendars such that each segment of consecutive work days has its own pricing problem. Moreover, the pricing problems for recurring segments are identical, allowing for aggregation. Note that \cite{coughlan2015} considers the decomposition by calendars in addition to the previously discussed decomposition by resources.

\subsubsection{Alternative Objectives}

\paragraph{Resource Investment Problems}
Resource investment problems (RIP) assume unlimited resource capacities but seek to minimize the peak (maximum) resource usage for a given project deadline \citep{mohring1984}. \cite{drexl2001} propose computing lower bounds for the RIP using Lagrangian relaxation and Dantzig-Wolfe decomposition. For the latter, \cite{drexl2001} propose a single-project version of the approach by \cite{deckro1991}, resulting in the same PSP pricing problem. \cite{coughlan2015} consider a multi-mode RIP (MRIP) with calendar constraints and dedicated resources. 

\subsection{Column Generation Algorithms for Project Scheduling}

The previously discussed studies propose several column generation algorithms. As depicted in Table \ref{table::CGAlgorithms}, they can be grouped into three categories: lower-bound computations, heuristics, and exact branch-and-price algorithms. 

\begin{table*}[!h]
\renewcommand*{\arraystretch}{0.75}
\begin{tabular}{llccc}
\toprule
\multirow{2}{*}{Publication} & Problem & Lower Bound  & Column Generation & Branch-and- \\
                             & Variant & Computations & Heuristic         &   Price                               \\\midrule 
 \cite{deckro1991}     &   RCMPSP      &              &  \Large$\bullet$                  &                                   \\
  \cite{brucker2000}     &   PRCPSP      &        \Large$\bullet$      &                    &                                   \\
  \cite{drexl2001}     &   RIP      &        \Large$\bullet$      &                    &                                   \\
  \cite{brucker2003}     &   MPRCPSP      &        \Large$\bullet$     &                    &                                   \\
  \cite{akkan2005}     &   MPSP      &             &      \Large$\bullet$               &                                   \\
  \cite{montoya2014}     &   MSPSP      &             &                   &    \Large$\bullet$                                \\
  \cite{coughlan2015}     &   MRIP      &             &                   &         \Large$\bullet$                            \\
  \cite{moukrim2015}     &   PRCPSP      &             &                   &              \Large$\bullet$                       \\
  \cite{volland2017}     &   RCPSP      &             &      \Large$\bullet$               &                                   \\
  \cite{wang2019}     &   RCPSP      &             &      \Large$\bullet$               &                                   \\
  \cite{vanEeckhout2020}     &   MRCPSP      &             &                 &     \Large$\bullet$                                  \\\bottomrule
\end{tabular}
\caption{Column generation algorithms for project scheduling.}\label{table::CGAlgorithms}
\end{table*}

\subsubsection{Lower Bound Computations}
\cite{brucker2000} use column generation to solve the linear relaxation of the antichain-based formulation of the PRCPSP for a given upper bound $T$ on the makespan. If this relaxation turns out to be infeasible, the approach proves that $T$ is a valid lower bound for the RCPSP. To further strengthen this bound, \cite{brucker2000} use constraint propagation techniques to exclude antichains prior to optimization by introducing additional precedence constraints. The subsequent study by \cite{brucker2003} extends this approach to multiple modes. \cite{drexl2001} compute lower bounds for the RIP solving the linear relaxation of a Dantzig-Wolfe reformulation via column generation. 

\subsubsection{Column Generation Heuristics}
\cite{deckro1991}, \cite{akkan2005}, \cite{volland2017}, and \cite{wang2019} propose heuristics known as \textit{price-and-branch}. Price-and-branch refers to generating an initial column pool by a number of pricing iterations followed by solving the RMP to integrality. The latter is usually done by enforcing integrality restrictions on the RMP and feeding it to a commercial solver. This approach often produces good primal solutions but does neither guarantee to find a feasible solution, assuming one exists, nor optimality. However, in contrast to most heuristics, price-and-branch algorithms also provide a dual bound, hence a lower bound, for the original problem. If the RMP is solved to optimality via column generation, it directly yields a dual bound. Otherwise, the relation between Lagrangian relaxation and column generation (they share the same subproblems) can be exploited to compute a dual bound; see \cite{desrosiers2024}, Chapter 6 for technical details.

\subsubsection{Branch-and-Price Algorithms}
\cite{montoya2014}, \cite{moukrim2015}, \cite{coughlan2015}, and \cite{vanEeckhout2020} present full branch-and-price approaches. We outline these approaches by detailing their branching schemes and other algorithmic components. A summary of the algorithmic features is given in Table \ref{table::LiteratureBP}.  

\begin{table*}[!h]
\centering
\footnotesize
\renewcommand*{\arraystretch}{0.75}
\begin{tabular}{llcccl}\toprule
\multirow{3}{*}{Publication}     & \multirow{3}{*}{Problem} & \multicolumn{3}{c}{Branching on}       &\multirow{3}{*}{Other Algorithmic Components}                     \\
\cmidrule(rl){3-5}
                                 &                          & Start  & Resource  & \multirow{2}{*}{Other}&                                                                      \\
                                 &                          & Times & Demands & &                                                                       \\
\midrule
\multirow{2}{*}{\cite{montoya2014}}           & \multirow{2}{*}{MSPSP}                    & \multirow{2}{*}{\Large$\bullet$}           &       \multirow{2}{*}{}           &  \multirow{2}{*}{}      & Primal heuristics, \\
         &                  &         &                 &       &  constraint propagation, and others \\
\multirow{2}{*}{\cite{coughlan2015}}          & \multirow{2}{*}{MRIP}                    & \multirow{2}{*}{\Large$\bullet$}           & \multirow{2}{*}{\Large$\bullet$}                &       & Primal heuristics,                   \\
         &              &          &               &       &  constraint propagation                          \\
\multirow{2}{*}{\cite{moukrim2015}}            &\multirow{2}{*}{PRCPSP}                   &             &                  & \multirow{2}{*}{\Large$\bullet$}     & Primal heuristics,                              \\
           &                 &             &                  &      & constraint 
 propagation                               \\
\cite{vanEeckhout2020} & MRCPSP                  &             & \Large$\bullet$               &       & Search space decomposition                                             \\
\bottomrule
\end{tabular}
\caption{Overview of branch-and-price algorithms for project scheduling.}\label{table::LiteratureBP}
\end{table*}

\paragraph{Branching Schemes} All proposed branching schemes branch on the original variables by adding constraints to the pricing problems; see Section \ref{sec::BranchingBasics}. \cite{montoya2014} propose to branch on the start times of activities. \cite{vanEeckhout2020} propose a branching framework tailored to the integrated shift and project scheduling problem. We focus on the branching rules used to obtain integral project schedules. To this end, \cite{vanEeckhout2020} present three branching rules based on resource demands, resulting in branching constraints similar to classical resource constraints. \cite{coughlan2015} propose a hierarchical branching scheme for the MRIP that first branches on resource demands before branching on start times. Lastly, \cite{moukrim2015} proposes a tailored branching scheme for the PRCPSP and the antichain formulation. As this approach is tailored to a special case, we refrain from discussing its details. 

\paragraph{Primal Heuristics} \cite{montoya2014}, \cite{coughlan2015}, and \cite{moukrim2015} utilize heuristics to find primal solutions that allow early pruning of suboptimal parts of the branching tree. \cite{moukrim2015} uses a not-specified heuristic to obtain an initial solution to the PRCPSP prior to solving the root node via column generation. Similarly, \cite{coughlan2015} uses a leveling heuristic for the RIP prior to column generation. Additionally, \cite{coughlan2015} use a list scheduling algorithm at each node of the branching tree that considers resource bounds and time windows implied by the current fractional solution. \cite{montoya2014} also use two primal heuristics, which are called at each node of the search tree. First, like \cite{coughlan2015}, they use a list scheduling algorithm. Second, they solve the restricted master problem to integrality. The latter is a common approach known as \textit{restricted master heuristic} \citep{sadykov2019}.

\paragraph{Search Space Reduction} \cite{montoya2014}, \cite{coughlan2015}, \cite{moukrim2015} use constraint propagation techniques to reduce the search space of the considered problems. All three tighten the time windows of activities and propagate the branching decisions regarding start times. Furthermore, \cite{moukrim2015} uses constraint propagation to determine additional precedence relations among activities.

\paragraph{Dual Bound Tightening}
As a complement to good primal bounds obtained via heuristics, tightening the dual bounds may allow early pruning of the branching tree. For this purpose, \cite{coughlan2015} proposes to improve lower bounds by using cutting planes based on the disaggregated precedence constraints proposed by \cite{christofides1987}. The cutting planes are added to the master problem, which creates new dual values that impact the objective function of the pricing problems. Also, \cite{montoya2014} seek to improve their lower bounds. However, instead of tightening the current relaxation, they compute alternative lower bounds based on a stable set problem.

\subsection{Positioning Our Work in the Literature}

The literature proposes Dantzig-Wolfe decompositions for several RCPSP variants as listed in Table \ref{table::LiteratureOverview}. Many of the proposed formulations share common characteristics, which have thus far been overlooked. Results that would justify the use of the proposed formulations and column generation algorithms (Table \ref{table::CGAlgorithms} and Table \ref{table::LiteratureBP}) compared to directly solving the compact formulation via a commercial solver are lacking. To this end, except for hardness proofs of the pricing problems in \cite{coughlan2015}, analytical results are missing. Also, numerical results benchmarking against commercial solvers are falling short.
\cite{deckro1991}, \cite{drexl2001}, \cite{akkan2005}, \cite{brucker2000, brucker2003}, and \cite{moukrim2015} do not benchmark against commercial solvers at all. \cite{volland2017}, \cite{wang2019}, and \cite{vanEeckhout2020} only compare upper bounds, not lower bounds and optimality gaps. Similarly, \cite{coughlan2015} only compare the number of solved instances and do not discuss the solution quality of suboptimal instances. Only \cite{montoya2014} considers upper bounds, lower bounds, and optimality gaps for benchmarking. \cite{montoya2014}, \cite{coughlan2015}, and \cite{vanEeckhout2020} claim to outperform commercial solvers. However, as shown in Table \ref{table::LiteratureBP}, they combine column generation with many other algorithmic components, such as constraint programming and problem-specific heuristics. Hence, it remains unclear if column generation or other algorithmic components drive the claimed superior performance.

In the remainder of this paper, we address the aforementioned gaps in the literature. First, we establish the missing connection between the proposed formulations by presenting a new formulation that generalizes most Dantzig-Wolfe decompositions from the literature (Section \ref{sec::DantzigWolfe}). Then, we study the structure of the Dantzig-Wolfe reformulated model and its impact on column generation. To this end, we analyze the master problem (Section \ref{sec::MasterProblem}) and the pricing problems (Section \ref{sec::PricingProblems}). The latter shed light on the relaxation strength of the Dantzig-Wolfe reformulated models. Additionally, we perform numerical experiments that give insights regarding the computational performance of the Dantzig-Wolfe decomposition. Lastly, we analyze two branching rules from the literature (Section \ref{sec::Branching}).

\section{Dantzig-Wolfe Decomposition for Non-Preemptive Scheduling}\label{sec::DantzigWolfe}

In this section, we define a project scheduling problem and corresponding mathematical models that generalize most non-preemptive RCPSP variants considered in the context of column generation. First, we define the problem and provide a compact formulation before presenting a Dantzig-Wolfe decomposition. For the latter, we provide a general description of the master problem and two ways to formulate the pricing problem. Furthermore, along with the problem definition and models, we discuss their special cases and how they relate to the literature. 

\subsection{Problem Definition}
We consider a multi-mode resource-constrained multi-project scheduling problem with a general objective function (MRCMPSP-GO). Formally, we consider a central decision-maker that simultaneously schedules a set of projects $\mathcal{P}$ competing for a shared pool of renewable resources $\mathcal{R}$ over discrete-time horizon $\mathcal{T}:=\{1, \ldots, T \}$. For conciseness, we do not consider non-renewable resources. However, the analytical results given in later sections also apply to non-renewable resources. Each project $i \in \mathcal{P}$ can be represented by an activity-on-node network $\mathcal{G}_{i} := \{\mathcal{V}_i, \mathcal{E}_i\}$, where the node set $\mathcal{V}_{i}$ represents the activities of project $i$, and the edges $\mathcal{E}_{i}$ represent precedence relationships between pairs of activities. Activities can be processed in multiple modes $\mathcal{M}:=\{1, 2,  \ldots, M\}$, where each activity $j \in \mathcal{V}_{i}$ has to be processed in of its modes $m \in \mathcal{M}_{ij} \subseteq \mathcal{M}$ for a duration of $d_{ijm}$ periods requiring $r_{ijmk}$ units of resource $k \in \mathcal{R}$ during each period of processing. Furthermore, an activity $j^\prime \in \mathcal{V}_{i}$ may only start after all its predecessors $j: (j, j^\prime) \in \mathcal{E}_{i}$ are finished and may not be interrupted once started. Starting activity $j$ of project $i$ in mode $m$ at period $t$ incurs a cost of $c_{ijmt}$ units. In contrast to classical RCPSPs, we consider the resource capacities to be decisions rather than parameters. The capacity assigned to resource $k \in \mathcal{R}$ is an integer between $\underline{R}_{k}$ and $\overline{R}_{k}$. Assigning capacities to resource $k$ incurs a cost of $c_{k}$ per unit. The objective aims at determining a schedule and resource capacities such that the sum of incurred costs is minimized. For appropriate choices of $c_{ijmt}$ and $c_{k}$, the objective covers many common scheduling objectives, such as minimizing the makespan, weighted tardiness, or resource investment costs.  Moreover, as depicted in Table \eqref{table::SpecialCases}, the MRCMPSP-GO contains most RCPSP variants considered in the context of column generation as special cases.

\begin{table*}[h]
\renewcommand*{\arraystretch}{0.75}
\centering
\begin{tabular}{llcccccc} \toprule
\multicolumn{1}{c}{\multirow{2}{*}{Publication}} & Problem & \multirow{2}{*}{$\vert \mathcal{P} \vert = 1$} & \multirow{2}{*}{$\vert \mathcal{M} \vert = 1$} & \multirow{2}{*}{$R_{k} = \overline{R}_{k}$} & \multirow{2}{*}{$\overline{R}_{k} = \infty$}  &  \\
\multicolumn{1}{c}{}                             & Variant &                        &                        &                                                &                                 \\
\midrule
\cite{deckro1991}                                          & RCMPSP  &                        & \Large$\bullet$                        & \Large$\bullet$                                                 &                                                              \\
\cite{drexl2001}                                            & RIP     & \Large$\bullet$                       & \Large$\bullet$                          &                                                                                  & \Large$\bullet$                               &     \\
\cite{akkan2005}                                            & MPSP    & \Large$\bullet$                        &                        &                                                                                 & \Large$\bullet$                                \\
\cite{coughlan2015}                                         & MRIP    & \Large$\bullet$                        &                        &                                                & \Large$\bullet$                                                                  \\
\cite{volland2017}                                          & RCPSP   & \Large$\bullet$                         & \Large$\bullet$                        &    \Large$\bullet$                                                                               &                        \\
\cite{wang2019}                                            & RCPSP   & \Large$\bullet$                        & \Large$\bullet$                         & \Large$\bullet$                                                &                                                          \\
\cite{vanEeckhout2020}                                 & MRCPSP  & \Large$\bullet$                          &                        &    \Large$\bullet$                                                                                &   \\
\bottomrule
\end{tabular}
\caption{Special cases of the MRCMPSP-GO.} \label{table::SpecialCases}
\end{table*}

\subsection{Compact Formulation}
Compact formulations for multi-project scheduling are single-project approaches because they merge the activity-on-node network of all projects into an artificial super-project $\mathcal{G}:= \{\mathcal{V}, \mathcal{E}\}$, where $\mathcal{V}:= \bigcup_{i \in \mathcal{P}} \mathcal{V}_{i}$ and $\mathcal{E}:= \bigcup_{i \in \mathcal{P}} \mathcal{E}_{i}$. For the super-project, the problem boils down to a single-project scheduling problem, and one can choose among its known formulations. We base our model on the well-known pulse disaggregated discrete-time (PDDT) formulation \citep{artigues2017}. The PDDT uses \textit{pulse variables} $x_{ijmt}$, which takes the value $1$ if activity $j$ of project $i$ is started in mode $m$ in period $t$, and $0$ otherwise. Furthermore, let $R_{k}$ be an integer variable representing the assigned capacity to resource $k$. Using this notation, we formulate the MRCMPSP-GO as follows:

\begin{align}
    z_{PDDT} = &\min \sum_{(i,j) \in \mathcal{V}} \sum_{m \in \mathcal{M}_{ij}}\sum_{t \in \mathcal{T}} c_{ijmt} x_{ijmt}  + \sum_{k \in \mathcal{R}} c_{k}R_{k} \label{MRCMPSP::Obj}\\
    \text{s.t.}~&\sum_{m \in \mathcal{M}_{ij}}\sum_{t \in \mathcal{T}} x_{ijmt} = 1 && \forall (i,j) \in  \mathcal{V} \label{MRCMPSP::AssCons}\\
    & \sum_{m \in \mathcal{M}_{ij}}\sum_{s = t- d_{ijm}+1}^{T}x_{ijms} + \sum_{m \in \mathcal{M}_{ij^\prime}}\sum_{s = 1}^{t} x_{ij^\prime ms} \leq 1  && \forall (i, j, j^\prime) \in \mathcal{E}, t \in \mathcal{T}\label{MRCMPSP::DisPredCons}\\
    &\sum_{(i,j) \in \mathcal{V}}\sum_{m \in \mathcal{M}_{ij}} \sum^{t}_{s = t - d_{ijm}+1} r_{ijmk} x_{ijms} \leq R_{k} && \forall k \in \mathcal{R},~ t \in \mathcal{T}\label{MRCMPSP::ResourceUsage}\\   
    & x_{ijmt} \in \{0,1\} && \forall (i,j)  \in \mathcal{V},~ m \in \mathcal{M}_{ij},~ t \in \mathcal{T} \label{MRCMPSP::XDomains}\\
    & R_{k} \in \{\underline{R}_{k}, \ldots, \overline{R}_{k}\} && \forall k \in \mathcal{R} \label{MRCMPSP::RDomains}
\end{align}
Objective \eqref{MRCMPSP::Obj} minimizes the costs of the schedule. Constraints \eqref{MRCMPSP::AssCons} ensure that each activity is assigned exactly one start time and mode. Constraints \eqref{MRCMPSP::DisPredCons} are disaggregated precedence constraints that ensure that an activity starts only after all its predecessors have been finished. Constraints \eqref{MRCMPSP::ResourceUsage} are resource constraints that  ensure that the resource demands do not exceed the assigned capacities. Finally, Constraints \eqref{MRCMPSP::XDomains} and \eqref{MRCMPSP::RDomains} define the variable domains.

At this point, we want to emphasize that the integer program \eqref{MRCMPSP::Obj}-\eqref{MRCMPSP::RDomains} can be used to solve the RCPSP variants depicted in Table \ref{table::SpecialCases}, but there also exist alternative formulations. For instance, the disaggregated precedence constraints \eqref{MRCMPSP::DisPredCons} can be aggregated into Constraints \eqref{MRCMPSP::AggPredCons}, resulting in a formulation known as pulse discrete-time (PDT) formulation.
\begin{align}
    & \sum_{m \in \mathcal{M}_{ij}}\sum_{t \in \mathcal{T}}(t+d_{ijm})x_{ijmt} \leq \sum_{m \in \mathcal{M}_{ij^\prime}}\sum_{t \in \mathcal{T}}tx_{ij^\prime mt}  && \forall (i, j, j^\prime) \in \mathcal{E} \label{MRCMPSP::AggPredCons}
\end{align}
The PDT requires fewer constraints but results in weaker linear relaxations \citep{artigues2017}. In fact, none of the studies using column generation for project scheduling consider the stronger disaggregated version. An alternative to a discrete-time formulation is a continuous time representation as considered by \cite{akkan2005}, which also gives a weaker formulation than the PDDT.

\subsection{Dantzig-Wolfe Decomposition} 
In the literature, Dantzig-Wolfe decomposition is used to decompose scheduling problems by subsets of activities. To generalize this decomposition, we partition the set of all activities $\mathcal{V}$ into $\mathcal{N}:= \{1,2, \ldots, N\}$ mutually exclusive and collectively exhaustive subsets $\mathcal{V}_{1}, \mathcal{V}_{2}, \ldots, \mathcal{V}_{N}$.
Let $\mathcal{E}_{n}$ be the set of precedence relationships corresponding to the activities in $\mathcal{V}_n$. Furthermore, we define $\Omega_n$ as the set of potential schedules for activities $\mathcal{V}_{n}$. Each schedule $\omega \in \Omega_{n}$ is associated with a number of parameters. The parameters $s_{ijn \omega}$ and $f_{ijn \omega}$ represent the start and finish times for activities $(i,j)  \in \mathcal{V}_{n}$ of partition $n$ using schedule $\omega$, respectively. Similarly, $c_{n \omega}$ denots costs incurred by schedule $\omega$. For the resources usage of schedule $\omega$ and activities $n \in \mathcal{V}_n$,  parameter $r_{knt \omega}$  denotes the demand for resources $k$ in period $t$ and  $\overline{r}_{kn \omega}$ denotes the peak demand for resource $k$.  Lastly, we define the binary variables $\lambda_{n\omega}$ that take the value $1$ if activities $n \in \mathcal{V}_n$ are processed according to schedule $\omega$ and 0 otherwise. Now we apply a Dantzig-Wolfe decomposition on the compact formulation \eqref{MRCMPSP::Obj}-\eqref{MRCMPSP::RDomains} and obtain the following master problem:
\begin{align}
    z_{DW} = &\min \sum_{n \in N} \sum_{\omega \in \Omega_n} c_{n\omega} \lambda_{n\omega} + \sum_{k \in \mathcal{R}} c_{k}R_{k} \label{DW::Obj}\\
    \text{s.t.}~&\sum_{\omega \in \Omega_n} \lambda_{n\omega} = 1 && \forall n \in \mathcal{N} && [\pi_{n}] \label{DW::Convexity}\\
    &\sum_{n \in \mathcal{N}: (i,j) \in \mathcal{V}_n} \sum_{\omega \in \Omega_n} f_{ijn \omega}  \leq \sum_{n \in \mathcal{N}: (i,j^\prime) \in \mathcal{V}_n} \sum_{\omega \in \Omega_n} s_{ij^\prime n \omega}   && \forall (i, j, j^\prime) \in \mathcal{E} \setminus{\bigcup_{n \in \mathcal{N}} \mathcal{E}_n} && [\pi_{ijj^\prime}] \label{DW::Precedence}\\
    &\sum_{n \in \mathcal{N}} \sum_{\omega \in \Omega_n} r_{knt \omega} \lambda_{n\omega} \leq R_{k} && \forall k \in \mathcal{R},~ t \in \mathcal{T} && [\pi_{kt}] \label{DW::ResourceUsage}\\
    &\sum_{n \in \mathcal{N}} \sum_{\omega \in \Omega_n} \overline{r}_{kn \omega} \lambda_{n\omega} \leq R_{k} && \forall k \in \mathcal{R}~ && [\pi_{k}] \label{DW::RBound}\\
    & \lambda_{n\omega} \in \{0,1\} && \forall n \in \mathcal{N}, ~ \omega \in \Omega_n\label{DW::Domains}\\
    & R_{k} \in \{\underline{R}_{k}, \ldots, \overline{R}_{k}\} && \forall k \in \mathcal{R} \label{DW::RDomains}
 \end{align} 

Objective \eqref{DW::Obj} is a straightforward reformulation of the original objective function \eqref{MRCMPSP::Obj}. Constraints \eqref{DW::Convexity} are convexity constraints that ensure that for each subset of activities $\mathcal{V}_{n}$, exactly one schedule is selected. Constraints \eqref{DW::Precedence} are reformulated precedence constraints \eqref{MRCMPSP::DisPredCons} but are limited to pairs of activities that do not belong to the same partition $\mathcal{V}_n$. Constraints \eqref{DW::ResourceUsage} are a reformulation of resource constraints \eqref{MRCMPSP::ResourceUsage}, and Constraints \eqref{DW::RBound} are auxiliary constraints that recover the peak resource usage of a column. Note that Constraints \eqref{DW::RBound} is redundant but allows for simplifications for a certain special case, which we discuss later.  Finally, Constraints \eqref{DW::Domains} and \eqref{DW::RDomains} define the decision variables. 

To solve the master problem via column generation, we restrict the master problem \eqref{MRCMPSP::RDomains}-\eqref{DW::Domains} to a small subset of schedules $\Tilde{\Omega}_n \subseteq \Omega_n$ and relax the integrality restrictions. For the resulting RMP, we denote $\pi^{\eqref{DW::Convexity}}_{n}$, $\pi^{\eqref{DW::Precedence}}_{ijj^\prime}$, $\pi^{\eqref{DW::ResourceUsage}}_{kt}$, and $\pi^{\eqref{DW::RBound}}_{k}$ as the dual values corresponding to the Constraints \eqref{DW::Convexity}, \eqref{DW::Precedence}, \eqref{DW::ResourceUsage}, and \eqref{DW::RBound}, respectively. Using the dual values, we can compute the reduced costs of a schedule $\hat{c}_{n}(\pi^{\eqref{DW::Convexity}}_{n}, \pi^{\eqref{DW::Precedence}}_{ijj^\prime}, \pi^{\eqref{DW::ResourceUsage}}_{kt},  \pi^{\eqref{DW::RBound}}_{k})$ as follows:
\begin{align}
   \hat{c}_{n}(\pi^{\eqref{DW::Convexity}}_{n}, \pi^{\eqref{DW::Precedence}}_{ijj^\prime}, \pi^{\eqref{DW::ResourceUsage}}_{kt}, \pi^{\eqref{DW::RBound}}_{k}) & =  \sum_{(i,j) \in \mathcal{V}_n} \sum_{m \in \mathcal{M}_{ij}}\sum_{t \in \mathcal{T}} (c_{ijmt}-c^{\eqref{DW::Precedence}}_{ijt} - c^{\eqref{DW::ResourceUsage}}_{ijmt}) x_{ijmt}  + 
     \sum_{k \in \mathcal{R}} (c_{k} - \pi^{\eqref{DW::RBound}}_{k})R_{k} - \pi_n \label{reducedCosts},\\
     \text{where }  &  c^{\eqref{DW::Precedence}}_{ijt} = \sum_{(i, j, j ^\prime) \in \mathcal{E} } \sum_{m \in \mathcal{M}_{ij^\prime}}\sum_{t \in \mathcal{T}}  t\pi_{ijj^\prime} -  \sum_{(i, j^\prime, j) \in \mathcal{E}}  \sum_{m \in \mathcal{M}_{ij}}\sum_{t \in \mathcal{T}}  (t+d_{ijm})\pi_{ij^\prime j} \text{ and} \notag\\
    & c^{\eqref{DW::ResourceUsage}}_{ijmt} = \sum_{k \in \mathcal{R}} \sum_{t \in \mathcal{T}}\sum_{m \in \mathcal{M}_{ij}} \sum^{t}_{s = t - d_{ijm}+1} \pi^{\eqref{DW::ResourceUsage}}_{kt} r_{ijmk}. \notag
\end{align}
To find an improving schedule with negative reduced costs, we can formulate pricing problems, one for each $n \in \mathcal{N}$, using the original variables $x_{jmt}$ and $R_{k}$. The pricing problems proposed in the literature fall into two categories: resource-unconstrained and resources-constrained. Most approaches propose the former, resulting in pricing problems without resource constraints that read as follows:
\begin{align}
    \min~& \eqref{reducedCosts}\notag\\
    \text{s.t.}~&\sum_{m \in \mathcal{M}_{ij}}\sum_{t \in \mathcal{T}} x_{ijmt} = 1 && \forall (i,j) \in \mathcal{V}_{n} \label{Pricing::AssCons}\\
    &  \sum_{m \in \mathcal{M}_{ij}}\sum_{s = t- d_{ijm}+1}^{T}x_{ijms} + \sum_{m \in \mathcal{M}_{ij^\prime}}\sum_{s = 1}^{t} x_{ij^\prime ms} \leq 1  && \forall (i, j, j^\prime) \in \mathcal{E}_{n}, t \in \mathcal{T}\label{Pricing::DisPredCons}\\
    & x_{ijmt} \in \{0,1\} && \forall (i,j) \in \mathcal{V}_{n},~ m \in \mathcal{M}_{ij},~ t \in \mathcal{T} \label{Pricing::XDomains}
\end{align}
For dedicated resources, also resource-constrained pricing problems have been proposed:
\begin{align}
   \min~& \eqref{reducedCosts} \notag \\
    \text{s.t.}~&\eqref{Pricing::AssCons}-\eqref{Pricing::XDomains} \notag\\
     & \sum_{(i, j) \in \mathcal{V}_n} \sum_{m \in \mathcal{M}_{ij}} \sum^{t}_{s = t - d_{ijm}+1} r_{ijmk} x_{ijms} \leq R_{k} && \forall k \in \mathcal{R},~ t \in 
    \mathcal{T} \label{Pricing::ResourceUsage}\\
    & R_{k} \in \{\underline{R}_{k}, \ldots, \overline{R}_{k}\} && \forall k \in \mathcal{R} \label{Pricing::RDomains}
\end{align}
Depending on the chosen pricing problem type and the considered RCPSP variant, the master and pricing problems may be simplified, as discussed in the following section.

\subsection{Special Cases and Simplifications}
Table \ref{table::DWLiterature} summarizes the simplified master and pricing problems for the column generation approaches and RCPSP variants from the literature. The simplifications result from the following three special cases.

\paragraph{Fixed Resource Capacities}
Most studies consider the case of classical resources with fixed capacities, i.e., $\underline{R}_k = \overline{R}_k$. For fixed capacities, we can substitute $R_k$ by $\overline{R}_k$ and drop Constraints \eqref{MRCMPSP::RDomains} and \eqref{DW::RBound}. Furthermore, removing \eqref{DW::RBound} from the master problem removes the dual values $\pi^{\eqref{DW::RBound}}_k$ and allows simplifying the reduced costs expression \eqref{reducedCosts}.

\paragraph{Decomposition by Projects} A decomposition of projects represents the special case that the partition $\mathcal{N}$ of activities $\mathcal{V}$ into subsets $\mathcal{V}_{n}$ coincide with the activities  $\mathcal{V}_i$ of project $i$, i.e., $\mathcal{V}_{n} = \mathcal{V}_{i}$ if $i = n$. For the decomposition by projects, the Constraints \eqref{DW::Precedence} and corresponding dual values $\pi^{\eqref{DW::Precedence}}_{ijj^\prime}$ disappear. For the single-project case $\vert \mathcal{P} \vert = 1$, a decomposition by projects results in a single pricing problem.

\paragraph{Decomposition by Dedicated Resources}
For dedicated resources, \cite{coughlan2015} and \cite{wang2019} propose to partition the set of activities $\mathcal{V}$ into subsets $\mathcal{V}_k$ such that all activities $(i,j) \in \mathcal{V}_k$ request resource $k$, i.e., $\mathcal{V}_{n} = \mathcal{V}_{k}$ if $n = k$. Furthermore, in combination with resource-constrained pricing problems, Constraints \eqref{DW::RBound} are the only required resource constraints for the master problem, and Constraints \eqref{DW::ResourceUsage} become redundant. 

\begin{table*}[!h]
\renewcommand*{\arraystretch}{0.75}
\centering
\begin{tabular}{lll} \toprule
Publication                       &                           Master Problem                                 & Pricing Problem                                \\
\midrule
\cite{deckro1991}                                                                            & $\min \eqref{DW::Obj}~\text{s.t.}~ \eqref{DW::Convexity}, \eqref{DW::ResourceUsage}, \eqref{DW::Domains}$                         & $\min \eqref{reducedCosts}~\text{s.t.}~\eqref{Pricing::AssCons}-\eqref{Pricing::XDomains}$ \\
\cite{drexl2001}                                             & $\min \eqref{DW::Obj} ~\text{s.t.}~  \eqref{DW::RDomains},\eqref{DW::Convexity}, \eqref{DW::ResourceUsage}, \eqref{DW::Domains} $& $\min \eqref{reducedCosts}~\text{s.t.}~\eqref{Pricing::AssCons}-\eqref{Pricing::XDomains}$    \\
\cite{akkan2005}                                             & $\min \eqref{DW::Obj} ~ \text{s.t.}~ \eqref{DW::Obj}-\eqref{DW::Precedence}, \eqref{DW::Domains}$   & $\min \eqref{reducedCosts}~\text{s.t.}~\eqref{Pricing::AssCons}-\eqref{Pricing::XDomains}$ \\
\cite{coughlan2015}                                                                                &  $\min \eqref{DW::Obj} ~\text{s.t.}~  \eqref{DW::Convexity},\eqref{DW::Precedence}, \eqref{DW::RBound}-\eqref{DW::RDomains} $& $\min \eqref{reducedCosts}~\text{s.t.}~\eqref{Pricing::AssCons}-\eqref{Pricing::RDomains}   $ \\
\cite{volland2017}                                         &  $\min \eqref{DW::Obj}~\text{s.t.}~ \eqref{DW::Convexity}, \eqref{DW::ResourceUsage}, \eqref{DW::Domains}$   &    $\min \eqref{reducedCosts}~\text{s.t.}~\eqref{Pricing::AssCons}-\eqref{Pricing::XDomains}$ \\
\cite{wang2019}                                            &  $\min \eqref{DW::Obj}~\text{s.t.}~ \eqref{DW::Convexity}, \eqref{DW::ResourceUsage}, \eqref{DW::Domains}$   & $\min \eqref{reducedCosts}~\text{s.t.}~\eqref{Pricing::AssCons}-\eqref{Pricing::ResourceUsage}   $     \\
\cite{vanEeckhout2020}                                   & $\min \eqref{DW::Obj}~\text{s.t.}~ \eqref{DW::Convexity}, \eqref{DW::ResourceUsage}, \eqref{DW::Domains}$   & $\min \eqref{reducedCosts}~\text{s.t.}~\eqref{Pricing::AssCons}-\eqref{Pricing::XDomains}$ \\
\bottomrule
\end{tabular}
\caption{Dantzig-Wolfe Decompositions for RCPSP variants considered in the literature.}\label{table::DWLiterature}
\end{table*}
\section{Master Problem Structure and Convergence} \label{sec::MasterProblem}

The structure of the master problem plays a crucial role in the efficiency of column generation algorithms. On the one hand, its structure may be exploited to design tailored column generation algorithms such as the all-integer column generation algorithms by \cite{ronnberg2014all}. On the other hand, its structure can negatively impact convergence as in the case of  \textit{dual noise} \citep{subramanian2008} and primal degeneracy \citep{vanderbeck2005implementing}. Dual noise describes the phenomenon where certain dual values are magnitudes larger than the coefficients of the original objective, leading to the generation of many columns that do not contribute toward convergence. Primal degeneracy refers to the phenomenon where two or more basic feasible solutions to a linear program correspond to the same solution, i.e., several bases represent the same corner point of the feasible region. Unfortunately, also the master problem \eqref{DW::Obj}-\eqref{DW::Domains} for non-preemptive scheduling suffers from degeneracy. Before discussing the issues primal degeneracy causes, let us formally show that linear relaxations of the proposed master problems for project scheduling problems (see Table \ref{table::DWLiterature}) are indeed prone to degeneracy. 

\begin{proposition}
   The restricted master problem for non-preemptive project scheduling problems is prone to primal degeneracy. \label{proposition::Degeneracy}
\end{proposition}
\begin{proof}
     For all variants, the master problem has $N + C$ rows, where $C$ is a positive integer representing the number of Constraints \eqref{DW::Precedence}-\eqref{DW::RBound}.     Consequently, a solution to the RMP consists of $N + C$ basic variables. However, a solution to the problem may require only $N$ non-zero variables, leaving up to $C$ basic variables with a value of zero. Replacing one of the zero basic variables with a non-basic variable results in a degenerate solution.
\end{proof}

The proof of Proposition \ref{proposition::Degeneracy} shows that degeneracy is caused by a disbalance between the number of non-zero variables in a solution and the number of constraints. In the case of a single-project RCPSP and a decomposition by projects ($N = 1$), a solution may need only one non-zero variable but has $N + \vert \mathcal{R} \vert \times \vert \mathcal{T} \vert $ constraints leading to a large disbalance. While this case is unlikely, the large number of variables (columns) can cause massive degeneracy even for solutions having only one non-zero variable less than constraints. In this case, we have one variable with a value of zero in the basis, which can be replaced by one of the many non-basic variables, each resulting in a degenerate solution.

Degeneracy causes three main problems in the context of branch-and-price. First, as the RMP is usually solved by the simplex algorithm, it causes degenerate pivots, slowing down its solving process. Moreover, we may repeat degenerate pivots in subsequent iterations because the RMP is solved in each column generation iteration. Second, it can lead to degenerate branching decisions in the branching process. Third, it leads to unreliable dual values that can misguide the pricing problems for the identification of improving columns. This can be explained from a dual point of view. For the dual of the RMP, column generation can be considered a cutting-plane algorithm, where adding a cut corresponds to adding a column to the primal problem. However, primal degeneracy implies the existence of multiple optimal solutions for the dual, and adding a new cut (column) may cut off only the current dual solution without improving the objective. In the worst case, the multiple dual solutions must be cut off individually. Hence, solving degenerate RMPs may require a large number of iterations and the generation of many columns. 

This behavior was empirically observed for project scheduling by \cite{volland2017} and \cite{drexl2001}. In the latter study, roughly half of the test instances have been excluded from the experiments because they required the generation of too many columns. We can also observe this behavior solving classical RCPSP instances from the PSPLIB \citep{kolisch1997}. Figure \ref{fig::Convergence} depicts the convergence (logarithmic axes) for solving the j30\_29\_4 instance with 30 activities. At each iteration, the most negative reduced cost column is added. After less than 20 iterations, we can observe the impact of primal degeneracy, where adding new columns barely improves the primal bound, and the progress stalls. In total, 1,765 columns are needed for convergence. This means we solve the RMP and pricing problem 1,765 times and generate as many columns just to solve the linear relaxation of the master problem, for an instance with only 30 activities.

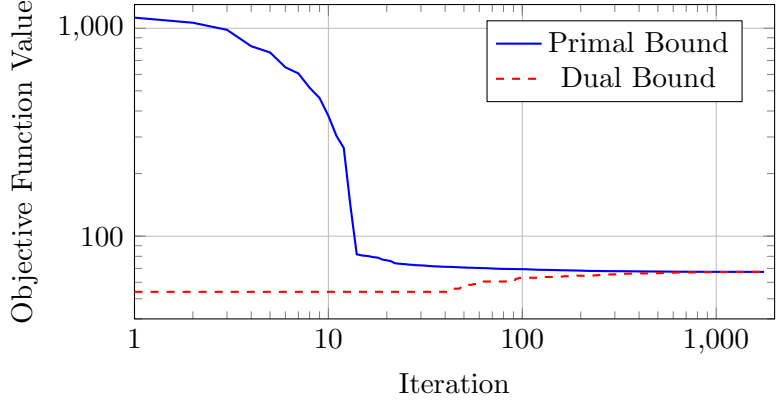
\begin{figure}[!h]
\centering
\begin{tikzpicture}
\begin{axis}[
    width=10cm, height=5.75cm,
    xlabel={Iteration},
    ylabel={Objective Function Value},
    xmode=log, 
    ymode=log, 
    log ticks with fixed point, 
    xmin=1, xmax=2000,
    ymin=40, ymax=1300,
    legend style={at={(0.95,0.95)}, anchor=north east},
    grid=major,
    title={}
]

\addplot[blue, solid, thick] table[x=Iter, y=UB, col sep=semicolon] {data.csv};
\addlegendentry{Primal Bound}

\addplot[red, dashed, thick] table[x=Iter, y=LB, col sep=semicolon] {data.csv};
\addlegendentry{Dual Bound}
\end{axis}
\end{tikzpicture}
\caption{Convergence for solving the linear relaxation of the master problem via column generation (logarithmic axes).}
\label{fig::Convergence}
\end{figure}

Degeneracy is not unique to project scheduling but is a common problem in column generation. Several mitigation strategies have been proposed in the literature. Most of these strategies fall under the umbrella of \textit{dual stabilization} and seek to manipulate the dual values to increase the likelihood of generating columns that contribute towards convergence \citep{pessoa2018}. Surprisingly, in the scope of project scheduling, dual stabilization is only used by \cite{wang2019}.
\section{Pricing Problem Structure and Lower Bounds}\label{sec::PricingProblems}
Empirically, it has been shown that Dantzig-Wolfe decomposition yields substantially stronger lower bounds for many problems, such as the vehicle routing problem \citep{costa2019}. These observations are theoretically grounded by Theorem \ref{theorem:LowerBounds}, which states that the lower bound of the reformulated problem is least as good as the bound from the original formulation. 
\begin{theorem}[\cite{desrosiers2024}, Chapter 4.1]\label{theorem:LowerBounds}
    Given the optimal solutions $z_{IP},~ z_{LP},~ z_{IMP},~ z_{LMP}$ for a compact integer program \ref{compact}, its linear relaxation LP, its integer master problem \ref{extensive}, and the linear relaxation of the master problem LMP, respectively. Then, $
        z_{LP} \leq z_{LMP} \leq  z_{IMP}= z_{IP}$.
\end{theorem}
\begin{proof}
Using the notation from Section \ref{sec::BPBasics}, we can prove this result by comparing the feasible regions of the different formulations:
    \begin{align}
        &\underbrace{\{x \in \mathbb{R}_{\geq 0}: Ax \leq a\} \cap \{x \in \mathbb{R}_{\geq 0}: Bx \leq b\}}_{\text{LP}}\supseteq \underbrace{conv(\mathcal{A}) \cap \{x \in \mathbb{R}_{\geq 0}: Bx \leq b\}}_{\text{LMP }} \supseteq  \underbrace{\mathcal{A} \cap\mathcal{B}}_{\text{IP and IMP}}, \notag\\
        &\text{where }\mathcal{A} = \bigcup_{n=1}^{N} \mathcal{A}_n. \notag
    \end{align}
\end{proof}

Theorem \ref{theorem:LowerBounds} shows that Dantzig-Wolfe decomposition can lead to smaller feasible regions and improved lower bounds. However, the lower bounds do not improve if the pricing problems possess a structure known as \textit{integrality property}.

\begin{definition}[\cite{guignard2003lagrangean}, integrality property]
    An integer program whose solutions are in $\mathcal{A}:= \{x \in \mathbb{Z}: Ax \leq a\}$ has the integrality property if and only if the convex hull $conv(\mathcal{A}) = \{x\in \mathbb{R}: Ax \leq a\}$.\label{definition::Integrality property}
\end{definition}
\begin{theorem}
    Let $z_{LP}$ and $z_{LMP}$ be the optimal solutions to the linear relaxation of a compact formulation and master formulation, respectively. Then, $z_{LP}$ is equal $z_{LMP}$ if the pricing problem has the integrality property, i.e., $z_{LP} = z_{LMP}$. \label{theorem::IntegralityProperty}
\end{theorem}

\begin{proof}
As done for Theorem \ref{theorem:LowerBounds}, we can prove this result by comparing the feasible regions of the different formulations: 
    \begin{align}
        &\underbrace{\{x \in \mathbb{R}_{\geq 0}: Ax \leq a\} \cap \{x \in \mathbb{R}_{\geq 0}: Bx \leq b\}}_{\text{LP}} \overset{\text{ Def. \ref{definition::Integrality property}}}{=} \underbrace{conv(\mathcal{A}) \cap \{x \in \mathbb{R}_{\geq 0}: Bx \leq b\}}_{\text{LMP }}, \text{where }\mathcal{A} = \bigcup_{n=1}^{N} \mathcal{A}_n. \notag
    \end{align}
\end{proof}

Theorem \ref{theorem::IntegralityProperty} was first proven by \cite{geoffrion1974} in the context of Lagrangian relaxation. Because Lagrangian relaxation and Dantzig-Wolfe decomposition share the same subproblems, the result applies to both. Moreover, both approaches yield the same lower bounds in general; see \cite{desrosiers2024}, Chapter 6 for technical details. In light of Theorem \ref{theorem::IntegralityProperty}, we analyze the structure of the different pricing problems proposed in the literature and discuss their impact on the quality of the lower bounds.

\subsection{Single-Mode without Resource Constraints}
\cite{deckro1991}, \cite{drexl2001}, and \cite{volland2017} consider single-mode PSPs without resource constraints for pricing. The polytope of the PSP has already been studied extensively decades ago, leading to the following result.
\begin{proposition}
    For $\vert \mathcal{M} \vert = 1$, the problem $\min \{\eqref{reducedCosts}: \eqref{Pricing::AssCons}-\eqref{Pricing::XDomains}\}$ possesses the integrality property. \label{proposition::PSP}
\end{proposition}

Proposition \ref{proposition::PSP}  has been proven in many ways, and we refer to \cite{mohring2001} for an overview of the different proof strategies. This result implies that to solve the pricing problem, it suffices to solve its linear relaxation. Moreover, as considered by \cite{drexl2001}, we can solve the pricing problem in pseudo-polynomial time by a combinatorial algorithm transforming the problem into a min-cut problem \citep{mohring2003}. What sounds like an advantage at first also implies, by Theorem \ref{theorem::IntegralityProperty}, that the lower bound is not strengthened via Dantzig-Wolfe decomposition. 

\cite{deckro1991}, \cite{drexl2001}, and \cite{volland2017} do not exactly consider $\min \{\eqref{reducedCosts}: \eqref{Pricing::AssCons}-\eqref{Pricing::XDomains}\}$ for pricing but its weaker PDT formulation based on aggregated precedence constraints, which does not possess the integrality property \citep{mohring2001}. Hence, they may improve their bounds via Dantzig-Wolfe decomposition. However, as the convex hull of the pricing problem with aggregated precedence constraints coincides with the feasible region of the linear relaxation of the stronger formulation with disaggregated precedence constraints, the lower bounds will only be as strong as the bounds from the compact PDDT formulation. Meaning the lower bound improvement they gain from reformulating the model and solving it via column generation can also be achieved by adding only the violated disaggregated precedence constraints \eqref{MRCMPSP::DisPredCons} as cuts. Moreover, as Constraints \eqref{MRCMPSP::DisPredCons} are essentially cover cuts, they are likely to be generated by commercial solvers automatically. 

To assess the practical performance of the different formulations, we compare them for the classical RCPSP with makespan objective solving 480 j30-instances from the PSPLIB \citep{kolisch1997}. To this end, we implement the models in Python and solve them using Gurobi 11.0.2. with standard settings on a machine that has an Intel(R) Xeon(R) W-1390P, 3.50GHz CPU with 8 cores. 

\begin{table}[!h]
\renewcommand*{\arraystretch}{0.75}
\centering
\begin{tabular}{lrrrrrr}\toprule
\multicolumn{1}{c}{\multirow{2}{*}{Formulation}} & \multicolumn{3}{c}{Improvement [\%]}                & \multicolumn{3}{c}{Runtime {[}sec{]}} \\ \cmidrule(rl){2-4} \cmidrule(rl){5-7}
                         & Min.   & Avg. (Std.)   & Max. & Min. & Avg. (Std.)          & Max.         \\
\midrule
LMP         & \multicolumn{1}{c}{-}                    & \multicolumn{1}{c}{-}   & \multicolumn{1}{c}{-}                                      & 0.33                  & 29.52 (74.06)      & 602.63      \\ 
PDDT-LP        & 0.00                     & 0.00 (0.00)  & 0.00                                       & 0.06               & 0.38 (0.56)      & 4.41           \\
PDDT-LP-Cut                                      & 0.00 & 6.81 (9.47)   & 59.15  &  0.86 & 4.95 (6.88)       & 68.67       \\
PDT-LP                                           & -10.14                        & -7.50 (-1.23)  & 0.00  & $<$0.01                     &  0.03 (0.03)       & 0.17             \\
PDT-LP-Cut                                      & 0.00                    & 4.98 (6.68)   & 40.64  & 0.02                   & 1.49 (3.13)        & 21.00          \\
 \bottomrule
\end{tabular}
\caption{Lower bound comparison for 480 RCPSP instances from the PSPLIB.}\label{table::LowerBoundsRCPSP}
\end{table}

Table \ref{table::LowerBoundsRCPSP} reports the lower bounds obtained from the linear relaxations of the different formulations, as well as the required runtimes. The lower bounds are given for each compact formulation relative to the lower bound obtained from the master problem, where LMP, PDDT-LP, and PDT-LP refer to the master problem $\min\{\eqref{DW::Obj}: \eqref{DW::Convexity}, \eqref{DW::ResourceUsage}, \eqref{DW::Domains}\}$, the compact disaggregated PDDT formulation, and the compact aggregated PDT formulation, respectively. Further, the suffix ``-Cut" indicates that we also use cutting planes, which are added automatically by the solver. We validate the reported results by comparing each pair of formulations using a two-sided Wilcoxon rank test. For all pairs, except the pair LMP and PDDT-LP, the lower bound improvements are significant for a significance level of 1\%.  Similarly, except for the pair LMP and PDDT-LP-Cut, the differences in runtimes are significant for a significance level of 1\%.

In line with the theoretical results given in Theorem \ref{theorem::IntegralityProperty} and Proposition \ref{proposition::PSP}, the LMP and PDDT-LP yield the same lower bounds. However, solving the LMP via column generation is, on average, 77 times slower. Adding cutting planes to the PDDT-LP yields significant improvements, and the PDDT-LP-Cut yields lower bounds that are, on average, 6.81\% stronger than the LMP. While the average runtime for the PDDT-LP-Cut is smaller than for the LMP, their differences are not significant because of large variances in runtime. The PDT-LP is, on average, 7.50\% weaker than the LMP and PDDT-LP concerning the lower bounds, which aligns with theoretical results by \cite{artigues2017} showing that the PDT-LP yields weaker relaxations than the PDDT-LP. Furthermore, because the PDT-LP is the most compact formulation out of the three, solving the PDT-LP is 984 and 12.7 times faster than the LMP and PDDT-LP, respectively. Again, we can improve over the LMP by considering cutting planes. The PDT-LP-Cut achieves average improvements of 4.98\% over the LMP while being 19.8 times faster to solve. Moreover, as shown by the results in Table \ref{table::MIPGaps}, the commercial solver can solve the integral PDT to optimality or near optimality in the same or less time it takes to solve the LMP via column generation. 

\begin{table}[!h]
\renewcommand*{\arraystretch}{0.75}
\centering
\begin{tabular}{llllll}\toprule
Optimality Gap & $\leq 0\%$ & $\leq 5\%$ & $\leq 10\%$ & $\leq 15\%$ & $\leq 20\%$ \\
\midrule
Instances      & 93.33\%        & 96.04\%         & 98.13\%          & 99.38\%          & 100.00\%        
\\ \bottomrule
\end{tabular}
\caption{Solution quality for solving the compact aggregated PDT formulation of 480 PSPLIB instances by a commercial solver with a time limit set to the time required to solve the relaxation of the master problem via column generation.}\label{table::MIPGaps}
\end{table}

We conclude that branch-and-price approaches resulting in single-mode pricing problems without resource constraints are not beneficial in terms of computational performance. Presumably unknowingly, \cite{volland2017} also presents empirical evidence for this claim as, for a given time limit, they observe stronger bounds from the compact formulation of the problem compared to its reformulated version. In this context, \cite{volland2017} point to Theorem \ref{theorem:LowerBounds} and argue that the observed lower bounds are surprising. Hence, we assume there is a missing awareness in the scheduling community of the integrality property (Proposition \ref{proposition::PSP}) and its impact on the lower bound (Theorem \ref{theorem::IntegralityProperty}). This assumption is reinforced by the study of \cite{drexl2001}, which compares lower bounds obtained via Lagrangian relaxation and column generation, which will not only give the same lower bound (Theorem \ref{theorem::IntegralityProperty}) but also yield no improvement beyond the compact disaggregated PDDT formulation. 

\subsection{Multi-Mode without Resource Constraints}
A natural extension of single-mode pricing problems without resource constraints is the multi-mode case. Despite \cite{akkan2005} claiming that this ``is a hard to solve problem in theory'', to the best of our knowledge, such theoretical results are missing. Moreover, it is not obvious whether the integrality property and polynomial time tractability from the single-mode case generalize to the multi-mode setting. Our analysis shows that at least the integrality property does not hold for multiple modes. 
\begin{proposition}
    For $\vert \mathcal{M} \vert \geq 2$, the problem $\min\{\eqref{reducedCosts}: \eqref{Pricing::AssCons}-\eqref{Pricing::XDomains}\}$ does not possess the integrality property. \label{proposition::MPSP}
\end{proposition}
We prove this claim by means of a counterexample given in \ref{app::CounterExample}. Proposition \ref{proposition::MPSP} shows that even for the strong disaggregate precedence constraints, the multi-mode variant does not possess the integrality property. This result gives hope that, in contrast to the single-mode case, Dantzig-Wolfe decompositions yield stronger lower bounds than the PDDT formulation. To assess the practical quality of the lower bounds, we solve the classical MRCMPSP for the different formulations using the MMLIB50-instances from the MMLIB \citep{vanpeteghem2014}. 

\begin{table}[!h]
\renewcommand*{\arraystretch}{0.75}
\centering
\begin{tabular}{lrrrrrr}\toprule
\multicolumn{1}{c}{\multirow{2}{*}{Difference}} & \multicolumn{3}{c}{Improvement [\%]}                & \multicolumn{3}{c}{Runtime {[}sec{]}} \\ \cmidrule(rl){2-4} \cmidrule(rl){5-7}
                         & Min.   & Avg. (Std.)   & Max. & Min. & Avg. (Std.)          & Max.         \\
\midrule
LMP                     & \multicolumn{1}{c}{-}                        & \multicolumn{1}{c}{-} & \multicolumn{1}{c}{-}  & 4.26                     & 2,715.75 (7,582.83)       & 63,173.55  \\
PDDT-LP                                           & -4.35                       & -0.42 (0.72)  & 0.00  & 15.85                     & 104.70 (173.86)      & 1,463.48           \\
PDDT-LP-Cut  & -0.35                       & 3.11 (4.24) & 18.26  & 13.39                     & 458.72 (557.55)       & 2,954.22\\

PDT-LP     & -12.43                       & -4.09 (4.07)  & 0.00  & 0.19                   & 0.40 (0.20)       & 1.09       \\
PDT-LP-Cut  & -5.94                        & 1.07 (4.15)  & 31.33  & 1.12                    & 15.82 (19.98)       & 111.19 \\ \bottomrule 
\end{tabular}
\caption{Lower bound improvements over the PDDT formulation for 200 MRCPSP instances from the MMLIB.}\label{table::LowerBoundsMRCPSP}
\end{table}

Preliminary experiments showed that solving the linear relaxations of the master problem via column generation is very time-consuming. Therefore, we speed up the convergence of column generation by considering all solutions, returned by one call of the pricing problem, that have negative reduced costs and by using aggregated precedence constraints in the pricing problem. While this leads to substantial speed-ups, we limit our experiments to 200 randomly sampled instances from the MMLIB for time reasons. 

Table \ref{table::LowerBoundsMRCPSP} reports the lower bounds for the different compact formulations relative to the lower bound obtained from the relaxed master problem, as well as their respective runtimes. The results are obtained using the same computational setup and naming convention as for the single-mode case. For all pairs, except the lower bounds of the pair LMP and PDDT-LP-Cut, the results are significant for a significance level of 1\%.

For the multi-mode case, the LMP yields lower bounds that are only 0.42\% stronger on average than the lower bounds from the PDDT-LP. However, solving the LMP via column generation takes several hours in many cases and is, on average, 26 times slower. Adding cutting planes to the PDDT-LP improves its strength, and the PDDT-LP-Cut yields lower bounds that are not significantly different from the lower bounds of the LMP. However, solving the PDDT-LP-Cut is 5.9 times faster than solving the LMP. As observed for the single-mode case, the PDT-LP is the weakest formulation but the fastest to solve. The PDT-LP yields, on average, 4.09\% weaker lower bounds than the LMP but is 6,789 times faster to solve. Furthermore, we can improve the PDT-LP beyond the LMP by considering cutting planes. The PDT-LP-Cut achieves average improvements of 1.07\% over the LMP while being more than 171 times faster to solve. Because stronger lower bounds in much shorter runtimes are obtained by solving the compact PDT formulation and adding cutting planes, we conclude that column-generation approaches using pricing problems without resource constraints are not computationally beneficial for multi-mode problems.

\subsection{Single- and Multi-Mode with Resource Constraints}

For the special case of dedicated resources as considered by \cite{coughlan2015} and \cite{wang2019}, we can decompose the problem such that the pricing problems are resource-constrained. For resource-constrained pricing problems, it is straightforward to show that they are not integral.

\begin{proposition}
    The pricing problems $\min\{\eqref{reducedCosts}: \eqref{MRCMPSP::RDomains},\eqref{Pricing::AssCons}-\eqref{Pricing::ResourceUsage}\}$ and $\min\{\eqref{reducedCosts}: \eqref{Pricing::AssCons}-\eqref{Pricing::ResourceUsage}\}$  do not possess the integrality property. \label{proposition::RCPSP}
\end{proposition}
\begin{proof} [Proof (strategy adopted from \cite{coughlan2015}):]
    For the RIP, i.e., $ \overline{R}_k = \infty$, the corresponding pricing problem $\min\{\eqref{reducedCosts}: \eqref{MRCMPSP::RDomains}, \eqref{Pricing::AssCons}-\eqref{Pricing::ResourceUsage}\}$ is itself a RIP, which is known to be NP-hard \citep{mohring1984}. Thus, it cannot possess the integrality property unless P = NP. Similarly, for the special case $\underline{R}_k = R_{k} = \overline{R}_k$ and $\vert \mathcal{M} \vert =\vert \mathcal{P} \vert = 1$, the corresponding pricing problem $\min\{\eqref{reducedCosts}: \eqref{Pricing::AssCons}-\eqref{Pricing::ResourceUsage}\}$ becomes an RCPSP, which is also NP-hard \citep{blazewicz1983}. 
\end{proof}

Proposition \ref{proposition::RCPSP} lets us assume that a decomposition considerably tightens the formulation as we get RIP or RCPSP pricing problems. These problems are not only known to be NP-hard but also to have weak relaxations, and hence, we expect to gain much by applying a Dantzig-Wolfe decomposition. The experiments by \cite{coughlan2015} confirm this intuition by showing that substantially stronger lower bounds are observed for RIPs in practice. However, solving the linear relaxation is much harder as we repetitively need to solve NP-hard pricing problems, which, as shown by \cite{coughlan2015}, become the bottleneck of the algorithm. However, they do not state how much time it needs to solve the linear relaxation of the compact and master formulation. Hence, the trade-off between runtime and lower bound strength remains unclear. Despite further problem-specific adaptations of their algorithm and many algorithmic enhancements, the branch-and-price algorithm by \cite{coughlan2015} has difficulties solving instances with 50 activities. Therefore, we assume that even for this special case, the increase in the lower bound is offset by a disproportionate increase in runtime.

\subsection{Other pricing problems}

Our Dantzig-Wolfe decomposition does not cover the MSRCPSP and the PRCPSP as special cases. Conveniently, for these variants, the structure of the pricing problems is well-known. For the MRCPSP, \cite{montoya2014} presents a formulation that has a min-cost flow pricing problem, which is polynomial-time solvable and has the integrality property. Hence, their Dantzig-Wolfe reformulation does not tighten the lower bounds, which explains why they use an alternative lower bound based on a stable set problem in their branch-and-price approach. For the PRCPSP, \cite{brucker2000, brucker2003}, and \cite{moukrim2015} use the alternative formulation by \cite{mingozzi1998}. This formulation can be considered a natural master formulation with pricing problems that are knapsack problems with side constraints. As the knapsack problem is NP-hard, it does not possess the integrality property. Hence, their problem formulation may yield tighter bounds than its compact formulation. While the latter was not formulated, it can be reverse engineered from the master formulation; see \cite{desrosiers2024}, Chapter 4 for technical details.

\subsection{Summary of Lower Bounds}
Table \ref{table::LowerBounds} summarizes the theoretical strength of the lower bounds for the different Dantzig-Wolfe decompositions that follow from Theorem \ref{theorem::IntegralityProperty} and the previous analysis of the pricing problems. The table compares the lower bounds from the linear relaxation of the compact (original) formulation $z_{LP}$ and the master problem $z_{LMP}$, respectively. Thereby, $z_{LP}$ and $z_{LMP}$ refer to the bounds obtained from the formulations given in the literature and not our PDDT formulation \eqref{MRCMPSP::Obj}-\eqref{MRCMPSP::RDomains}. As the bounds from the literature may be weaker, when appropriate, we also compare against the bounds from the linear relaxation of the PDDT formulation denoted by $z_{PDDT-LP}$. The results show that for non-preemptive RCPSP variants resulting in single-mode pricing problems without resource constraints, Dantzig-Wolfe decomposition does not strengthen lower bounds beyond the compact disaggregated PDDT formulation. While the bounds for the decompositions resulting in multi-mode pricing problems without resource constraints may improve, our numerical results show that the improvements are of a rather theoretical nature, as only marginal improvements are observed. For the special variants of dedicated resources and preemptive activities, numerical results from the literature indicate that the master formulations are substantially stronger.
\begin{table}[!h]
\renewcommand*{\arraystretch}{0.75}
\centering
\begin{tabular}{lll} \toprule
Publication                       & Problem Variant&                           Lower Bounds                                                              \\
\midrule
\cite{deckro1991}   & RCPSP                                                                         & $z_{LP} \leq z_{LMP} = z_{PDDT-LP}$                         \\
\cite{brucker2000}       & PRCPSP                                       & $z_{LP} \leq z_{LMP}$ \\
\cite{drexl2001}               & RIP                                & $z_{LP} \leq z_{LMP} = z_{PDDT-LP}$ \\
\cite{brucker2003}       & MPRCPSP                                        & $z_{LP} \leq z_{LMP}$ \\
\cite{akkan2005}           & MPSP                                    & $z_{LP} \leq z_{PDDT-LP}\leq z_{LMP}$   \\
\cite{coughlan2015}    & MRIP                                                                              &  $z_{LP} \leq z_{PDDT-LP}\leq z_{LMP}$\\
\cite{montoya2014}       & MSRCPSP                                        & $z_{LP} = z_{LMP}$ \\
\cite{moukrim2015}        & PRCPSP                                       & $z_{LP} \leq z_{LMP}$ \\
\cite{volland2017}            & RCPSP                               &  $z_{LP} \leq z_{LMP} = z_{PDDT-LP}$    \\
\cite{wang2019}                & RCPSP                              &  $z_{LP} \leq z_{PDDT-LP}\leq z_{LMP}$   \\
\cite{vanEeckhout2020}           & MRCPSP                          & $z_{LP} \leq z_{PDDT-LP}\leq z_{LMP}$   \\
\bottomrule
\end{tabular}
\caption{Theortical lower bound strength.}\label{table::LowerBounds}
\end{table}

\section{Branching and Lower Bounds} \label{sec::Branching}
So far, we have analyzed the lower bounds obtained by solving the linear relaxation of the master problem without the consideration of branching constraints. In other words, we have analyzed the lower bounds at the root node of the branching tree. However, the branching strategies proposed in the literature add additional constraints to the pricing problems, which may break their structure and thus strengthen the relaxations of later nodes in the branching tree. Specifically, we mean strengthening the lower bounds of the master formulation $z^{(\cdot)}_{LMP}$ relative to the compact formulation $z^{(\cdot)}_{LP}$ when the same branching constraints $(\cdot)$ are added to both formulations. In the following, we analyze the impact of branching on start times and resource demands. We do not consider the branching on interval orders as proposed by \cite{moukrim2015} because it is tailored to preemptive scheduling based on the alternative formulation by \cite{mingozzi1998} and  does not directly apply to classical formulations. Furthermore, we consider single-mode PSP pricing problems without resource constraints because the lower bounds for other pricing problems are potentially already tighter at the root node. 

\subsection{Branching on Start Times} \label{subsec::BranchingStartTimes}

Branching on the start times of activities as proposed by \cite{montoya2014} and \cite{coughlan2015} can be formally described as follows. Let $s^{\ast}_{ij}$ be the start time of activity $(i,j) \in \mathcal{V}_{n}$ for a fractional solution. First, we remove all columns $\lambda_{n \omega} \in \Tilde{\Omega}_n$ such that $s_{ijn\omega} > s^{\ast}_{ij}$ or $s_{ijn\omega} \leq s^{\ast}_{ij}$, depending on the branching direction, from the master problem. Then, to avoid regenerating the removed columns, we add the following branching constraints to the pricing problems:
\begingroup
\allowdisplaybreaks
\begin{align}
     &\text{down-branch}: x_{ijmt} = 0~ \forall t \in \mathcal{T}: t > \lfloor s^{\ast}_{ij} \rfloor &&  \text{up-branch}:
      x_{ijmt} = 0~ \forall t \in \mathcal{T}: t < \lceil s^{\ast}_{ij} \rceil  \label{BranchingStartTimes} 
\end{align}
\endgroup
Constraints \eqref{BranchingStartTimes} can be enforced efficiently in practice by adjusting the bounds of the variables $x_{ijmt}$ or removing them, which maintains the structure of the problem.

\begin{proposition}
      For $\vert \mathcal{M} \vert = 1$, adding Constraints \eqref{BranchingStartTimes} to pricing problems $\min \{\eqref{reducedCosts}: \eqref{Pricing::AssCons}-\eqref{Pricing::XDomains}\}$, does not improve the lower bound strength of the master formulation beyond the compact formulation, i.e., $z^{\eqref{BranchingStartTimes}}_{LMP} = z^{\eqref{BranchingStartTimes}}_{LP}$.\label{proposition::BranchingStartTimes}
\end{proposition}
\begin{proof}
    We prove this claim by showing that adding Constraints \eqref{BranchingStartTimes} to the pricing problem  maintains its integrality property. 
    Let $\mathfrak{P} := \{x \in [0,1]: \eqref{Pricing::AssCons}-\eqref{Pricing::DisPredCons}\}$ be the polytope corresponding to problem $\min \{\eqref{reducedCosts}: \eqref{Pricing::AssCons}-\eqref{Pricing::XDomains}\}$ and $\Pi(\mathfrak{P})$ be set of vertices of polytope $\mathfrak{P}$.  Adding Constraints \eqref{BranchingStartTimes} to $\mathfrak{P}$ results in a new polytope $\mathfrak{P}^\prime \subset \mathfrak{P}$. The constraints correspond to setting a subset of variables of a $[0,1]$-polytope to zero, which may remove but does not add new vertices, and we get $\Pi(\mathfrak{P}^\prime) \subseteq \Pi(\mathfrak{P})$. Finally, considering Proposition \ref{proposition::PSP}, which says that $\Pi(\mathfrak{P}) \in \mathbb{Z}$ because it is the convex hull of a set of integer points, we get $\Pi(\mathfrak{P}^\prime) \subseteq  \Pi(\mathfrak{P}) \in \mathbb{Z} \implies \text{conv}(\mathfrak{P}^\prime \cup \mathbb{Z}) =\text{conv}(\mathfrak{P}^\prime)$. Now, $z^{\eqref{BranchingStartTimes}}_{LMP} = z^{\eqref{BranchingStartTimes}}_{LP}$  follows from Theorem \ref{theorem::IntegralityProperty}.
\end{proof}

Proposition \ref{proposition::BranchingStartTimes} shows that branching on start times does not improve the relaxation strength of the master problem compared to the compact formulation. This means if we use branching on start times for both the master and compact formulation, we will generate the same branching tree.

\subsection{Branching on Resource Demands} \label{subsec::BranchingResourceDemands}

\cite{coughlan2015} and \cite{vanEeckhout2020} propose three branching rules based on the resource demands $r_{knt \omega}$ of the columns (schedules) $\lambda_{n \omega} \in \Tilde{\Omega}_{n}$. For some fractional solution, let $r^{\ast}_{knt}$ be the resource demand of activities $(i,j) \in \mathcal{V}_n$ for resource $k$ and period $t$, the three suggested branching rules use the following constraints:
\begingroup
\allowdisplaybreaks
\begin{align}
&\text{down-branch}: \sum_{k \in \mathcal{R}}\sum_{t \in \mathcal{T}}r_{knt} \leq \floor[\Bigg]{\sum_{k \in \mathcal{R}}\sum_{t \in \mathcal{T}}r^{\ast}_{kt}} &&
   \text{up-branch}:\sum_{k \in \mathcal{R}}\sum_{t \in \mathcal{T}}r_{knt} \geq \ceil[\Bigg]{\sum_{k \in \mathcal{R}}\sum_{t \in \mathcal{T}}r_{knt} r^{\ast}_{knt}} \label{BranchingResourceDemands1}\\
&\text{down-branch}: \sum_{k \in \mathcal{R}}r_{knt} \leq \floor[\Bigg]{\sum_{k \in \mathcal{R}}r^{\ast}_{knt}} &&
   \text{up-branch}: \sum_{k \in \mathcal{R}}r_{knt} \geq \ceil[\Bigg]{\sum_{k \in \mathcal{R}}r^{\ast}_{knt}}\label{BranchingResourceDemands2} \\
&\text{down-branch}:r_{knt} \leq \floor[\big]{r^{\ast}_{knt}} &&
   \text{up-branch}:r_{knt} \geq \ceil[\big]{ r^{\ast}_{knt}} \label{BranchingResourceDemands3}, \\
&\text{where }  r_{knt} = \sum_{(i, j) \in \mathcal{V}_n} \sum_{m \in \mathcal{M}_{ij}} \sum^{t}_{s = t - d_{ijm}+1} r_{ijmk} x_{ijms}  \notag
\end{align}
\endgroup
The branching constraints \eqref{BranchingResourceDemands1} and \eqref{BranchingResourceDemands2} can be considered non-renewable resource constraints that only apply to the multi-mode case.  In contrast, Constraints \eqref{BranchingResourceDemands3} are classical resource constraints that can also be applied to the single-mode variant. Our analysis shows that adding just one constraint of type \eqref{BranchingResourceDemands3} breaks the structure of the pricing problem.

\begin{proposition}
    Adding branching constraints \eqref{BranchingResourceDemands3} to the pricing problems $\min \{\eqref{reducedCosts}: \eqref{Pricing::AssCons}-\eqref{Pricing::XDomains}\}$ strengthens the master formulation beyond the compact formulation, i.e., $z^{\eqref{BranchingResourceDemands3}}_{LP} \leq z^{\eqref{BranchingResourceDemands3}}_{LMP}$.  \label{proposition::BranchingResourceDemands}
\end{proposition}

\begin{proof}[Proof]
     From Theorem \ref{theorem:LowerBounds} we know $z^{\eqref{BranchingResourceDemands3}}_{LP} \leq z^{\eqref{BranchingResourceDemands3}}_{LMP}$. It remains to show that there exists at least some case, where $z^{\eqref{BranchingResourceDemands3}}_{LP} < z^{\eqref{BranchingResourceDemands3}}_{LMP}$. As we show in \ref{app::SpecialCase}, the pricing problem $\min \{\eqref{reducedCosts}: \eqref{Pricing::AssCons}-\eqref{Pricing::XDomains}\}$ becomes NP-hard after adding only one constraint of type \eqref{BranchingResourceDemands3}. Hence, the feasible region of the linear relaxation of the compact formulation $LP$ is a proper subset of the linear relaxed master problem $LMP$. This implies the existence of a point $x^\ast \in LP \land x^\ast \notin LMP$. For the right choices of the cost coefficients $c_{ijmt}$ and $c_{k}$, the point $x^\ast$ is the unique optimum for $LP$ and we get $z^{\eqref{BranchingResourceDemands3}}_{LP} < z^{\eqref{BranchingResourceDemands3}}_{LMP}$.
\end{proof}

Proposition \ref{proposition::BranchingResourceDemands} shows that branching on resource demands improves the strength of the extensive over the compact formulation because adding resource constraints makes the pricing problem NP-hard. As a consequence, the lower bounds of the extensive formulation may be stronger. Thus, at least theoretically, branching on resource demands is stronger than branching on the start times of activities. This was also observed in practice by \cite{coughlan2015}, inspiring the hierarchical branching scheme that first branches on resource demands before branching on start times. Among the three different resource-based branching rules \eqref{BranchingResourceDemands1}, \eqref{BranchingResourceDemands2}, and  \eqref{BranchingResourceDemands3}, no substantial difference in runtime could be observed by \cite{vanEeckhout2020}.

\section{Synthesis of Results}\label{sec::Synthesis}

For classical single-mode project scheduling problems, the Dantzig-Wolfe decomposition yields a pricing problem that can be solved as a linear program (Proposition \ref{proposition::PSP}). However, the linear relaxation of the master problem is as weak as the relaxation of the compact PDDT formulation (Theorem \ref{theorem::IntegralityProperty} and Proposition \ref{proposition::PSP}). Also, branching on start times does not improve the relaxation strength compared to the compact formulation (Proposition \ref{proposition::BranchingStartTimes}). Hence, a branch-and-price algorithm that branches on start times generates the same branching trees as a classical branch-and-bound algorithm applied to the compact formulation. Consequently, branch-and-price will only be faster if we can solve the nodes of the branching tree faster. However, our results show that solving the linear relaxation of the master problem is much more time-consuming than solving the relaxation of the compact formulation. While we can potentially reduce the size of the branching tree by branching on resource demands (Proposition \ref{proposition::BranchingResourceDemands}), this will break the structure of the pricing problem, resulting in an NP-hard problem for which no efficient solution approach exists.  Moreover, solving only the relaxation of the master problem may requires many iterations due to the strong degeneracy (Proposition \ref{proposition::Degeneracy}), and in most cases, takes more time than solving a compact formulation to proven optimality. Therefore, strengthening the formulation by branching on resource demands will not yield a competitive algorithm.

For multiple modes, applying Dantzig-Wolfe decomposition strengthens the linear relaxation (Proposition \ref{proposition::MPSP}). However, our numerical experiments show that the stronger relaxations translate only into marginal lower-bound improvements, which are offset by a disproportionate increase in runtime. Therefore, as for the single-mode case, column generation is not competitive compared to directly solving a compact formulation via a commercial solver.

For the special case of dedicated resources, Dantzig-Wolfe decomposition can be applied to decompose the problem by dedicated resources. This decomposition strengthens the relaxation strength (Proposition \ref{proposition::RCPSP}) and substantially improves the lower bounds as observed by \cite{coughlan2015}. Despite the improved lower bounds, the numerical study by \cite{coughlan2015} shows that even with many algorithmic enhancements, such a decomposition struggles to solve instances of 50 activities to optimality. Solving the pricing problems, which are NP-hard scheduling problems, is the bottleneck for this decomposition. Alleviating this bottleneck would require an efficient approach for solving the pricing problem. As the pricing problem has the same structure as the original formulation, such an algorithm could be used to directly solve the compact formulation. 

Lastly, for the preemptive case, an alternative model formulation based on antichains can be used. This formulation utilizes many of the benefits of branch-and-price. First, the pricing problems are NP-hard knapsack problems, which result in stronger relaxations; \cite{moukrim2015} reports new best lower bounds for the PSPLIB instances that even exceed lower bounds from the non-preemptive setting. Second, the knapsack problem is well-studied, and thus, many tailored solution algorithms exist. Third, the decomposition breaks symmetries and allows for the aggregation of pricing problems. While this suggests that the branch-and-price for the PRCPSP performs well, computational results comparing against compact formulations are missing.

\section{Conclusion and Future Research}\label{sec::Conculsion}
We have provided results showing that the structure of Dantzig-Wolfe reformulated multi- and single-mode project scheduling problems is unfavorable for column generation. Therefore, we recommend not applying the column generation approaches from the literature for classical project scheduling problems but encouraging research on alternative formulations and Dantzig-Wolfe decompositions. This research should aim to develop formulations that have hard pricing problems for which efficient solution approaches exist. Furthermore, it should address degeneracy by formulating a master problem that is not prone to degeneracy or by taking countermeasures such as utilizing dual stabilization methods.

For special variants of the RCPSP, such as preemptive activities and dedicated resources, branch-and-price can lead to stronger relaxations due to NP-hard pricing problems. While some papers have shown that column generation applied to these special variants yields strong lower bounds, it remains unclear if the lower bound strength justifies the increased complexity of column generation. Future work could explore the trade-off between runtime and relaxation strength to assess the merit of branch-and-price for special cases. Furthermore, identifying other RCPSP variants with properties desirable for branch-and-price constitutes an avenue for future research. 

\section*{Funding}

The first author, Maximilian Kolter, is funded by Deutsche Forschungsgemeinschaft (DFG, German Research Foundation) - GRK 2201/2 - Projektnummer 277991500.
\bibliographystyle{elsarticle-harv} 
\bibliography{cas-ref}

\appendix
\section{Proof of Proposition \ref{proposition::MPSP}}\label{app::CounterExample}

\begin{proof}
 Consider the pricing problem $\min \{\eqref{reducedCosts}: \eqref{Pricing::AssCons}-\eqref{Pricing::XDomains}\}$ for a project $i \in \mathcal{P}$ with five activities ($\mathcal{V}_i = \{1, \ldots, 5\}$), two modes ($\vert \mathcal{M} \vert = 2$), activity durations as given in Table \ref{table::ActivityCharateristics}, precedence relations $\mathcal{E}_{i} = \{(1,2), (1,3), (2,5), (3,4), (4,5)\}$, a time horizon of eight periods ($\vert \mathcal{T} \vert = 8$), and the following costs:
    \begin{align}
        c_{ijmt} \begin{cases}
        1, & \text{if } m = 1, \\
       1, & \text{if } j = 4 \land m = 2 \land (t = 3 \lor t = 4), \\
       0, & \text{otherwise.}
    \end{cases}
    \end{align}
    The optimal solution to the linear relaxation  $\min \{\eqref{reducedCosts}: \eqref{Pricing::AssCons}-\eqref{Pricing::DisPredCons},x \in \mathbb{R}\}$ is $x_{111} = x_{121}= x_{213}= x_{222}= x_{322}= x_{415}= x_{416}= x_{517}= x_{526} = \frac{1}{2}$ and zero for all other variables. Because the solution is fractional it must hold that $conv(\{\eqref{Pricing::AssCons}-\eqref{Pricing::XDomains}\}) \neq \{\eqref{Pricing::AssCons}-\eqref{Pricing::DisPredCons},~x \in \mathbb{R}\}$, meaning the problem is not integral.
\end{proof}

\begin{table}[!h]
\renewcommand*{\arraystretch}{0.65}
    \centering
    \begin{tabular}{cccccc} \toprule
     \multirow{2}{*}{Mode $m$} & \multicolumn{5}{c}{Activity $j$}                \\\cmidrule(rl){2-6} 
     & 1 & 2 & 3 & 4 & 5\\   \midrule  
     1 & 1 & 1 & 1 & 1 & 1\\
     2 & 2 & 5 & 3 & 2 & 2\\ \bottomrule
    \end{tabular}
    \caption{Activity durations.}
    \label{table::ActivityCharateristics}
\end{table}

\section{Special Case for Proof of Propositon \ref{proposition::BranchingResourceDemands}}\label{app::SpecialCase}
As shown in the following, the pricing problem $\min \{\eqref{reducedCosts}: \eqref{Pricing::AssCons}-\eqref{Pricing::XDomains}\}$ becomes NP-hard after adding only one branching constraint \eqref{BranchingResourceDemands3} because it contains the NP-hard multiple-choice knapsack problem as special case. Consider a single-mode single-project instance ($\vert \mathcal{M} \vert = \vert \mathcal{P} \vert = 1$)  without precedence constraints ($\mathcal{E} = \emptyset$), a single resource ($\vert \mathcal{R} \vert = 1$), and activities that all have a duration of one unit ($d_{ijm} = 1 ~\forall j \in \mathcal{V}$).
Suppose we add branching constraint \eqref{BranchingResourceDemands3} for the resource $k^\prime$, period $t^\prime$, and right-hand-side $\lfloor r^{\ast}_{k^\prime t^\prime} \rfloor$ to the pricing problem. If we drop all indices for which the corresponding sets have cardinality one and if we transform the problem into a maximization problem by multiplying all cost coefficients in the objective by minus one, the problem reads as follows:
\begin{align}
        \max~&  \sum_{j \in \mathcal{V}}\sum_{t \in \mathcal{T}} \label{MCKP::obj} (c^{\eqref{DW::ResourceUsage}}_{jt}-c_{jt}) x_{jt}  + \pi_n   \\
        \text{s.t.}.~&\sum_{t \in \mathcal{T}} x_{jt} = 1 && \forall j \in \mathcal{V} \\    
        &\sum_{j \in \mathcal{V}} r_{j} x_{jt^\prime} +  \sum_{j \in \mathcal{V}}\sum_{j \in \mathcal{T}\setminus{\{t^\prime\}}} 0\times x_{jt} \leq \lfloor r^{\ast}_{t^\prime} \rfloor \\
        & x_{jt} \in \{0,1\} && \forall j \in \mathcal{V},~ t \in \mathcal{T}  \label{MCKP::xDomains}
\end{align}
Now, it is easy to see that the problem \eqref{MCKP::obj}-\eqref{MCKP::xDomains} is a multiple-choice knapsack problem with $\vert \mathcal{T} \vert$ classes of items $j \in \mathcal{V}$ that need to be packed into a knapsack with capacity $\lfloor r^{\ast}_{t^\prime} \rfloor$. Where, each item $j \in \mathcal{V}$ of class $t \in \mathcal{T}$ has a profit of $(c_{jt}-c^{\eqref{DW::ResourceUsage}}_{jt})$ units and a weight of $r_{j}$ units if $t = t^\prime$ and $0$ units otherwise.

\end{document}